\documentclass[11pt,twoside,letter]{article}
\usepackage{amsthm, amsmath, amssymb, amsfonts}
\usepackage{mathtools}

\usepackage{bbm}
\usepackage{mathrsfs}

\usepackage{graphicx}

\usepackage{epstopdf}

\usepackage[shortlabels]{enumitem}

\usepackage[colorlinks=true, pdfstartview=FitV, linkcolor=blue, citecolor=blue, urlcolor=blue]{hyperref}
\usepackage[usenames]{color}

\usepackage{url}
\makeatletter\def\url@leostyle{%
	\@ifundefined{selectfont}{\def\UrlFont{\sf}}{\def\UrlFont{\scriptsize\ttfamily}}} \makeatother

\usepackage[margin=1.0in, letterpaper]{geometry}

\newtheorem{theorem}{Theorem}
\newtheorem{proposition}[theorem]{Proposition}
\newtheorem{lemma}[theorem]{Lemma}
\newtheorem{assumption}{Assumption}
\theoremstyle{definition}
\newtheorem{definition}[theorem]{Definition}
\newtheorem{corollary}[theorem]{Corollary}
\newtheorem{example}[theorem]{Example}
\theoremstyle{remark}
\newtheorem{remark}[theorem]{Remark}

\numberwithin{equation}{section}
\numberwithin{theorem}{section}

\definecolor{Red}{rgb}{0.9,0,0.0}
\definecolor{Blue}{rgb}{0,0.0,1.0}

\def\cB{\mathcal{B}}

\def\cP{\mathcal{P}}

\def\cR{\mathcal{R}}

\def\cT{\mathcal{T}}

\def\cV{\mathcal{V}}

\def\bE{\mathbb{E}}

\def\bN{\mathbb{N}}

\def\bP{\mathbb{P}}
\def\bQ{\mathbb{Q}}
\def\bR{\mathbb{R}}
\def\bS{\mathbb{S}}

\def\bX{\mathbb{X}}

\def\sF{\mathscr{F}}

\def\sL{\mathscr{L}}

\def\sS{\mathscr{S}}

\def\mT{\mathsf{T}}

\newcommand{\set}[1]{\{#1\}}            
\renewcommand{\mid}{\;|\;}              
\newcommand{\norm}[1]{ \| #1 \| }       

\usepackage{tikz}
\usepackage{multirow}
\usepackage{subfig}
\usepackage{caption}
\DeclareMathOperator*{\vsup}{v-sup} 
\DeclareMathOperator*{\Vsup}{V-sup} 

\DeclareMathOperator*{\vinf}{v-inf} 
\newcommand{\inprod}[2]{\langle #1, #2 \rangle}
\DeclareMathOperator*{\cl}{cl} 

\title{Vector-valued robust stochastic control}

\makeatletter
\def\and{%
\end{tabular}%
\begin{tabular}[t]{c}}%
\def\@fnsymbol#1{\ensuremath{\ifcase#1\or a\or b\or c\or
		d\or e\or f\or g\or h\or i\else\@ctrerr\fi}}
\makeatother

\author{
	Igor Cialenco\,\thanks{Department of Applied Mathematics, Illinois Institute of Technology
		\newline \hspace*{1.45em}  10 W 32nd Str, Building RE, Room 220, Chicago, IL 60616, USA
		\newline \hspace*{1.45em}  Emails: \url{cialenco@iit.edu}, URL: \url{http://cialenco.com}
		\vspace{0.5em}} \quad and  
	\and
	Gabriela Kov\'a\v{c}ov\'a\,\thanks{
Department of Mathematics, University of California, Los Angeles, 
\newline \hspace*{1.45em} Portola Plaza 520, Los Angeles, CA 90095, USA \newline \hspace*{1.45em} Email: \url{kovacova@ucla.edu}, URL: \url{https://sites.google.com/view/kovacova/homepage} 
	}
}

\date{ {\small 
		First Circulated and this version: June 28, 2024\\
}}

\begin{document}

	\maketitle

	\vspace{-2em}
	
	
	{\footnotesize
		\begin{tabular}{l@{} p{350pt}}
			\hline \\[-.2em]
			\textsc{Abstract}: \ &

We study a dynamic stochastic control problem subject to Knightian uncertainty with multi-objective (vector-valued) criteria. Assuming the preferences across expected multi-loss vectors are represented by a given, yet general, preorder, we address the model uncertainty by adopting a robust or minimax perspective, minimizing expected loss across the worst-case model. For loss functions taking real (or scalar) values, there is no ambiguity in interpreting supremum and infimum. In contrast to the scalar case, major challenges for multi-loss control problems include properly defining and interpreting the notions of supremum and infimum, as well as addressing the non-uniqueness of these suprema and infima. To deal with these, we employ the notion of an ideal point vector-valued supremum for the robust part of the problem, while we view the control part as a multi-objective (or vector) optimization problem. Using a set-valued framework, we derive both a weak and strong version of the dynamic programming principle (DPP) or Bellman equations by taking the value function as the collection of all worst expected losses across all feasible actions. The weak version of Bellman's principle is proved under minimal assumptions. To establish a stronger version of DPP, we introduce the rectangularity property with respect to a general preorder. We also further study a particular, but important, case of component-wise partial order of vectors, for which we additionally derive DPP under a different set-valued notion for the value function, the so-called upper image of the multi-objective problem. Finally, we provide illustrative examples motivated by financial problems.

These results will serve as a foundation for addressing time-inconsistent problems subject to model uncertainty through the lens of a set-valued framework, as well as for studying multi-portfolio allocation problems under model uncertainty.

\\[0.5em]
\textsc{Keywords:} \ &  set-valued control; model uncertainty;  stochastic robust control; multi-objective criteria; Bellman's principle; dynamic programming; rectangularity property; Knightian uncertainty.   \\
\textsc{MSC2020:} \ &  Primary 90C39; Secondary  90C29, 93E03  \\[1em]
\hline
\end{tabular}

	}
	
\section{Introduction}
Model uncertainty refers to the challenge of accurately modeling the dynamics of the underlying stochastic system. This uncertainty may arise from various sources, such as incomplete information, ambiguity or presence of unobservable factors. It was first discussed by Knight~\cite{Knight1921} and is also referred as Knightian uncertainty.

A particularly important field extensively involving Knightian uncertainty is the control of stochastic systems subject to model uncertainty. Here, the controller does not know the true law of the underlying stochastic process a priori but knows it belongs to a known family of probability laws. As such, the controller faces not only the randomness of the controlled system, but also the Knightian uncertainty. A significant body of literature has addressed the challenge of model uncertainty, particularly driven by issues in finance and economics where flawed models can lead to erroneous investment decisions, ineffective risk management strategies, and inaccurate pricing of financial instruments. 

One approach to tackle this challenge is through robust optimization, which seeks controls that perform well across various possible models. Other approaches include adaptive control, Bayesian control, adaptive-robust control, and strong robust control. For a brief overview of these approaches, we refer the reader to  \cite{BCCCJ2019}. Additionally, recent advancements in approximation methods using machine learning have led to the development of new efficient methods for solving these control problems \cite{KrachTeichmann2024}. Substantial advancements have been achieved, encompassing highly abstract scenarios such as general multidimensional state spaces, discrete and continuous time frameworks, finite and infinite horizon settings, and finite or infinite model uncertainty spaces. These developments have also established connections to partial differential equations (PDEs) and backward stochastic differential equations (BSDEs). 
However, it is important to note that all these achievements assume that  the loss or reward function takes real one dimensional values; we refer to this as  scalar case.

On the other hand, there is a growing body of literature on multi-valued or set-valued stochastic control problems, with prominent applications in mathematical finance and economics. In \cite{RudloffUlus2020}, the authors study certainty equivalent and utility indifference pricing for incomplete preferences using vector optimization, stemming from \cite{Nau2006}, which is dedicated to the representation of incomplete preferences. Multi-portfolios and markets with transaction costs in a dynamic setup are studied in \cite{FeinsteinRudloff2013a,LoehneRudloff2014,FeinsteinRudloff2015,FeinsteinRudloff2019,AraratFeinstein2020} using multi-valued or set-valued risk measures. A novel approach to dealing with (scalar) time-inconsistent stochastic control problems by viewing them as set-valued problems was introduced in \cite{KovacovaRudloff2019}, which in particular was applied to portfolio optimization problems; see also \cite{CKR2021} for applications to the maximization of dynamic acceptability indices.

The literature on multi-objective optimization problems under model uncertainty is predominantly focused on the static case. Specifically, a robust approach to static multi-objective problems under model uncertainty has been explored through two main streams of research. The first stream, exemplified by works such as \cite{KuroiwaLee2012,FliegeWerner2014}, addresses robust multi-objective problems by employing a robustified objective. The second stream, represented by studies like \cite{EhrgottEtAl2014,IdeEtAl2014}, utilizes a set-optimization approach.
For an overview of existing results in static robust multi-objective (or vector) optimization, we refer the reader to the surveys \cite{IdeSchoebel2015,WiecekDranichak2016}.

To the best of our knowledge, this study is the first attempt to investigate (dynamic) stochastic control problems subject to model or Knightian uncertainty, involving multi-objective (vector-valued) criteria. These results will serve as a foundation for addressing time-inconsistent problems subject to model uncertainty through the lens of a set-valued framework, as well as for studying multi-portfolio allocation problems under model uncertainty.

Our approach assumes that the decision maker's preference across expected (multi-)loss vectors is represented by a given, yet general, preorder. We address model uncertainty by adopting a robust or minimax perspective, minimizing expected loss across the worst-case model.
For loss functions taking real (or scalar) values, there is no ambiguity in interpreting supremum and infimum. In contrast, one major challenge for multi-loss control problems is to properly define and interpret the notion of supremum and infimum. Another key difficulty is that usually these suprema and infima are not unique. 

To deal with these obstacles,  first, we employ the notion of \textit{ideal point vector-valued supremum} of a collection of vectors in $\bR^d$ with respect to a preorder, see \cite[Example 1.8]{Loehne2011}. We also introduce its dynamic, or conditional, version. Using this notion of the supremum for the robust part, we treat the control part of the problem as a multi-objective (or vector) optimization problem with respect to the preorder. 

The \textbf{second} novel contribution is derivation of a \textit{version of dynamic programming principle} (DPP) or Bellman equations.  
For multi-objective or set-optimization control problems there is no universal definition of a value function. 
Recent works on this topic have illustrated the advantage of utilizing set-valued mappings  as value functions, see e.g.~\cite{KovacovaRudloff2019,FeinsteinEtAl2022,HamelVisetti2020,FeinsteinRudloff2017,IseriZhang2021,IseriZhang2023}. In broad terms, the DPP corresponds to order relations between sets.  Similar to \cite{KovacovaRudloff2019}, we take the value function to be the collection of all worst expected losses across all feasible actions and, without any additional assumptions, derive  the corresponding backward recursive inclusions for the robust problem; see Theorem~\ref{th:Bellman1} and Theorem~\ref{th:Bellman0-1}.

For the scalar, one dimensional, stochastic robust control problems, it is well understood that the DPP hinges on the so called rectangularity property of the set of probability measures describing the model uncertainty, and we refer to  \cite{Shapiro2016} and references therein. 
The  \textbf{third} key contribution of our study is the introduction of \textit{rectangularity property} with respect to a  general preorder. Leveraging this we prove a stronger version of Bellman's principle of optimality; see Theorem~\ref{th:Bellman3} and Theorem~\ref{th:Bellman4}.

Using a series of results, we show that while the vector-valued supremum may not be unique, the Bellman's principle is invariant with respect to chosen supremum. \textbf{Fourth}, we further study a particular, but important, case of component-wise partial order of vectors for which we additionally derive DPP under a second popular set-valued notion of a value function, the so-called upper image of the multi-objective problem.
Finally, we provide some illustrative examples.

The paper is organized as follows. In Section~2 we set the stage by introducing the underlying stochastic model and briefly reviewing the fundamental concepts associated with vector preorders (Section~\ref{sec:preorders}) and vector optimization problems (Section~\ref{sec:VOP}). Section~\ref{sec:sup} is dedicated to the notion of supremum of collection of vectors with respect to a preorder. We discuss several important properties of this supremum, such as existence, uniqueness and monotonicity. Section~\ref{sec:dyn-prog-preorder} is dedicated to Bellman's principle of optimality for the main stochastic control problem. Starting with precise formulation of the problem, we define the set-valued candidate for the value function, and derive a weak version of DPP; Theorem~\ref{th:Bellman1} and Theorem~\ref{th:Bellman0-1}. Also here we introduce the notion of rectangularity with respect to a preorder, and prove strong versions of DPP, Theorem~\ref{th:Bellman3} and Theorem~\ref{th:Bellman4}. In Section~\ref{sec:component-wise} we discuss an important and natural preorder--component-wise partial order, where we also provide a DPP in terms of upper-sets of the value function. Last section is devoted to examples.

We are confident that the concepts outlined in this manuscript will form the basis for exploring general stochastic control problems under model uncertainty for multi-valued or set-valued criteria. Among some important avenues for imminent exploration, we highlight: 
other approaches to model uncertainty such as adaptive control, adaptive-robust control, Bayesian control; more general objectives such as risk-reward or dynamic risk; development of numerical solutions and computationally feasible algorithms to solve these problems. 
Furthermore, in light of recent advancements in reinforcement learning methodology, it would be advantageous to investigate the aforementioned problems within the context of Markov Decision Processes (MDP). Unlike the scalar case, however, there exists no straightforward mapping between MDP and our framework.

\section{Preliminaries}\label{sec:MDP-setup}

Let $(\Omega, \sF)$ be a measurable space,  and $T\in \bN$ be a fixed time horizon. Let $\cT=\{0,1,2,\ldots,T\}$, $\cT'=\{0,1,2,\ldots,T-1\}$, and $\Theta\subset\bR^d$ be a non-empty set, which will play the role of the parameter space throughout. For simplicity, in this work we assume that $\Theta$ is finite.

On the space $(\Omega, \sF)$ we consider a random process $S=\{S_t,\ t\in \cT\}$ taking values in some measurable space $(\bS,\sS)$ with $\bS\subset \bR^m$. We postulate that this process is observed, and we denote by ${\mathbb {F}}^S=(\sF^S_t,t\in \cT)$ its natural filtration. The randomness of process $S$ is derived from a stochastic factor process $Z$. The (true) law of $Z$  is unknown, and assumed to belong to a parameterized family of probability distributions  $\mathbf{Q}(\Theta):=\{\bQ_\theta:  \theta\in \Theta\}$ on $(\Omega, \sF)$. We assume that the elements of $\mathbf{Q}(\Theta)$ are absolutely continuous with respect to a reference probability measure $\bP$. 
We denote by $\bE^\theta$, respectively $\bE_t^\theta$, the expectation, respectively the conditional expectation given $\sF_t$, with respect to the probability $\bQ_\theta.$ The true, but unknown,  law of $Z$ will be denoted by $\bQ_{\theta^*}$, so that $\theta ^*\in \Theta$ is the unknown true parameter. In what follows all equalities, inequalities and inclusions will be understood in $\bP$-a.s. sense.

In the sequel, we postulate that the process $S$ follows the controlled dynamics:
\begin{equation}\label{eq:dyn}
	S_{t+1}=F(t,S_t,\varphi_t,Z_{t+1}),\ t\in \cT',\ S_0=s,
\end{equation}
where $\varphi_t$ is a time $t$ control taking value in a compact set $A$, $Z=(Z_t,\, t\in\cT)$ is an $\bR^k$-valued random sequence whose law under each measure $\bQ_\theta$ is known, and $F: \cT \times \bS \times A \times \bR^k \to \bS$ is a measurable function. We will denote by $S^\varphi$ the controlled process $S$ corresponding to the control $\varphi$. With slight abuse of notations we will also use the notation $S_{t+1}^{\varphi_t}:=F(t,S_t,\varphi_t,Z_{t+1})$. The set of admissible controls $\varphi=(\varphi_t,\ldots, \varphi_{T-1})$ starting at time $t$, and given value $S_t$, is denote by  $\mathfrak{A}^t(S_t)$, while $\mathfrak{A}_t(S_t)\subset A$ is the set of one step time $t$ admissible controls $\varphi_t$, given $S_t$.

We now consider a multi-loss/multi-cost function $\ell: \bS \rightarrow \bR^d$. Hence, for each admissible strategy $\varphi\in\mathfrak{A}$ and unknown parameter $\theta\in \Theta$, the expected loss $\bE^{\theta} \left[ \ell(S^\varphi_T) \right]$ is multivariate, and thus to optimize this expected loss, one has to clarify how we compare (random) vectors, see the next subsection.

Our aim is to apply the robust approach to model uncertainty for a problem with multiple objectives. That is, we want to understand and solve a problem of the form
\begin{align}\label{eq:rob_multi}
\textrm{`` }
\inf_{\varphi\in\mathfrak{A}} \sup_{\theta\in \Theta} \bE^{\theta} \left[ \ell(S^\varphi_T) \right] \textrm{ ''}.
\end{align}
To interpret this problem meaningfully, we first need to clarify how we understand suprema over vectors that we address next. 

\subsection{Vector preorder and partial order}\label{sec:preorders}
A \textit{vector preorder} $\preceq$ on $\bR^d$ is reflexive ($x \preceq x$), transitive ($x \preceq y$ and $y \preceq z$ imply $x \preceq z$) and compatible with the vector space structure of $\bR^d$ ($x \preceq y$ implies $x + z \preceq y + z$ and $\alpha x \preceq \alpha y$ for $\alpha \geq 0$ and $z\in \bR^d$). A \textit{vector partial order} is a vector preorder that is additionally antisymmetric ($x \preceq y$ and $y \preceq x$ imply $x = y$).

An order relation $\preceq$ generates an \textit{ordering cone} $C_{\preceq} = \{ x \in \bR^d : 0 \preceq x \}$. Similarly, a set $C \subseteq \bR^d$ can be used to define an order relation $\preceq_C$ as $x \preceq_C y$ whenever $y-x \in C$. Recall that a cone $C$ is called \textit{pointed} if $C \cap (-C) = \{0\}$ and it is called \textit{solid} if it has a non-empty interior.  Next we summarizes some well-known facts about order relations and their ordering cones: 
\begin{itemize}
\item It holds $C_{\preceq_C} = C$. Moreover, if $\preceq$ is compatible with the vector structure, then $\preceq_{C_\preceq}$ coincides with $\preceq$.
\item If $\preceq$ is a vector preorder, then $C_\preceq$ is a convex cone with $0 \in C_\preceq$. If $\preceq$ is a vector partial order, then $C_\preceq$ is a pointed convex cone. 
\item If $C$ is a convex cone with $0 \in C$, then $\preceq_C$ is a vector preorder. If $C$ is a pointed convex cone, then $\preceq_C$ is a vector partial order. 
\end{itemize}

In the sequel, we fix a vector preorder $\preceq$ and simply denote by $C$ the corresponding ordering cone. Compatibility with the vector space is assumed throughout, and often we simply say preorder and partial order instead of vector preorder and vector partial order. 
Preorder (or partial order) $\preceq$ on $\bR^d$ naturally extends to a preorder (or partial order) on the space $\sL_t(\bR^d):=L^1(\Omega, \sF_t, \bP; \bR^d)$ of $\sF_t$-measurable  and integrable (classes of equivalence of) random variables with values in $\bR^d$. In the following, $\sL_t(C) \subseteq \sL_t(\bR^d)$ denotes the set of $\sF_t$-measurable \ random vectors with values in $C$. We use $\sL_t(C)$ to define order relation $\preceq^t$, for $X,Y\in \sL_t(\bR^d)$ we say that $X \preceq^t Y$ if and only if $Y-X \in \sL_t(C)$,  which implies $X(\omega) \preceq Y(\omega)$, $\bP$-a.s. It is easily verified that properties of $\preceq$ are inherited by the order relation $\preceq^t$ (in the $\bP$-a.s. and conditional sense) and $\preceq^t$ is a preorder (or partial order) on $\sL_t(\bR^d)$.

In order to compare sets of vectors, a given preorder on the space of vectors can be used to define various order relations on the space of sets, see e.g.~\cite{HamelEtAl2015}. Within this work we focus on two canonical set order relations generated by the vector preorder $\preceq$ and its ordering cone $C$. 
For $A, B \subseteq \bR^d$ we define
\begin{align*}
A \preccurlyeq B \quad &\Leftrightarrow \quad B \subseteq A + C, \\
&\text{and} \\
A \curlyeqprec B \quad &\Leftrightarrow \quad A \subseteq B - C.
\end{align*}
Similarly, we define  the order relations $\preccurlyeq^t$ and $\curlyeqprec^t$ for subsets of $\sL_t(\bR^d)$.

\begin{lemma}\label{lemma:order_exp}
The preorder is compatible with the conditional expectation operator, that is, if $X, Y \in \sL_{t+1} (\bR^d)$ and  $X \preceq^{t+1} Y$, then $\bE^\theta_t[X] \preceq^t \bE^\theta_t[Y]$, for any $\theta \in \Theta$. 
\end{lemma}
\begin{proof}
If  $X \preceq^{t+1} Y$, then $Z := Y-X \in \sL_{t+1}(C)$. Since $C$ is convex, then the set generated by each coordinate of $C$ is convex in $\bR$. Recall that the conditional expectation preserves values in a convex set, and thus each coordinate of $\bE_t^\theta[Z]$ takes values in the same set where the corresponding coordinate of $Z$ takes values. Consequently, we have that $\bE_t^\theta[Y-X] = \bE_t^\theta [Z] \in \sL_t(C)$.    
\end{proof}

One of the simplest and most widely used partial orders on $\bR^d$ is generated by the ordering cone $C=\bR^d_+ = \set{x\in\bR^d : x_i\geq 0 , i=1,\ldots, d}$. We denote this partial order simply by $\leq$ since it corresponds to coordinate-wise comparison of vectors. Section~\ref{sec:component-wise} is dedicated to this partial order.

\subsection{Supremum of vectors}\label{sec:sup}
In simplified terms, the robust control approach is based on the following idea: for each choice of a control determine the worst-case model (the supremum in \eqref{eq:rob_multi}), then optimize across all feasible controls (the infimum in \eqref{eq:rob_multi}). For loss functions $\ell$ taking real values, there is no ambiguity in interpreting supremum and infimum. In contrast, for multi-loss problem, with $\ell$ taking values in $\bR^d$, one first needs to clarify how to interpret a supremum of a collection of (random) vectors. In this regards, vector (and set) optimization, including (deterministic) robust multi-objective optimization, literature, provides two canonical ways to think about the supremum:
\begin{itemize}
\item \textit{as a vector} with the property of being the lowest upper bound of the collection of vectors.  We will refer to this vector-valued notion of the supremum as \textit{ideal point} supremum, properly defined below. As shown later, even for finite collection of vectors, often no member of the collection is an upper bound for all other members. In terms of model uncertainty, a vector-valued supremum (assuming it exists) will often not correspond to any market model, but rather be an idealized scenario, hence the name ideal point. This approach corresponds to the robustified objective proposed  in~\cite{KuroiwaLee2012, FliegeWerner2014} for static robust multi-objective problems.

\item \textit{as sets} using set optimization framework. An upper bound would be a so-called lower set; if one considers finite collections, then a set of maximizers can also serve the purpose. Recall that order relations extended to set relations allow us to compare sets; minimizing a set-valued supremum would yield a set optimization problem. This is the basis of the second approach to robust multi-objective optimization explored in~\cite{EhrgottEtAl2014,IdeEtAl2014}.
\end{itemize} 
These two approaches to define supremum are illustrated in Figure~\ref{fig:0}. 

\begin{figure}[h]
\centering
\subfloat{
\centering
\includegraphics[width = 0.45\textwidth]{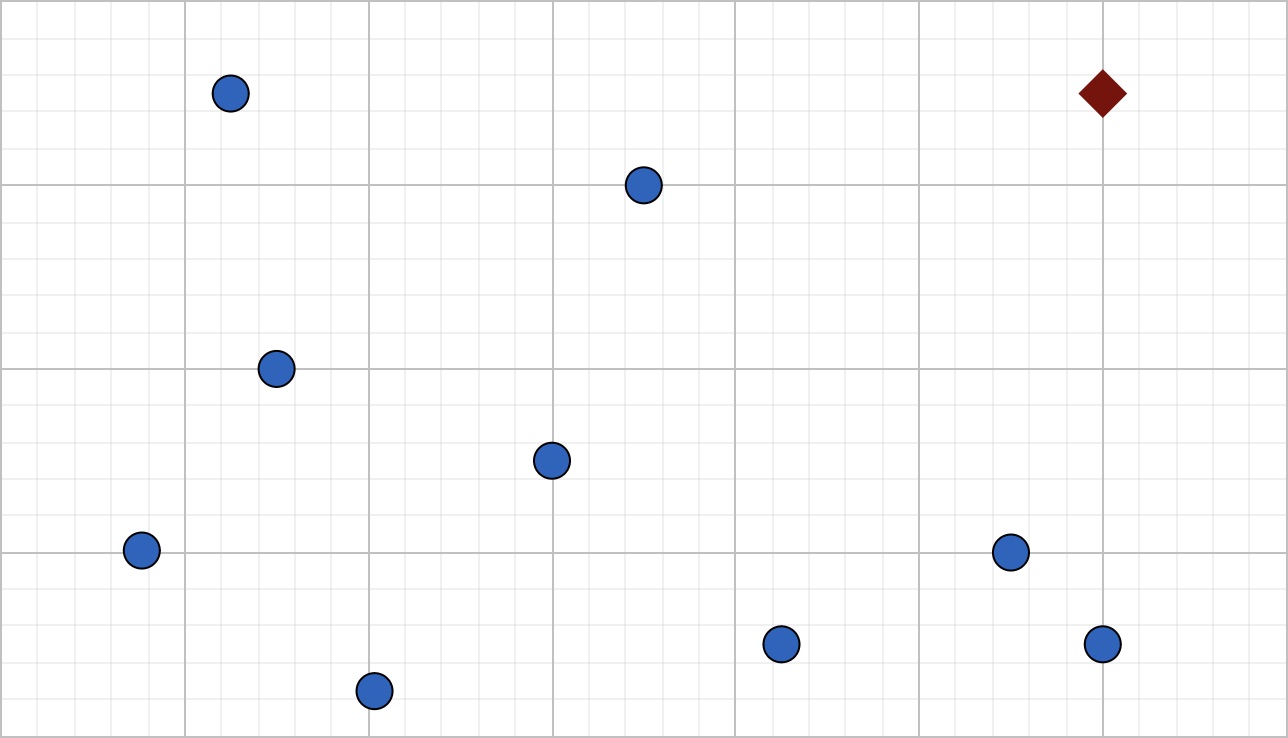}
}
~
\subfloat{
\includegraphics[width = 0.45\textwidth]{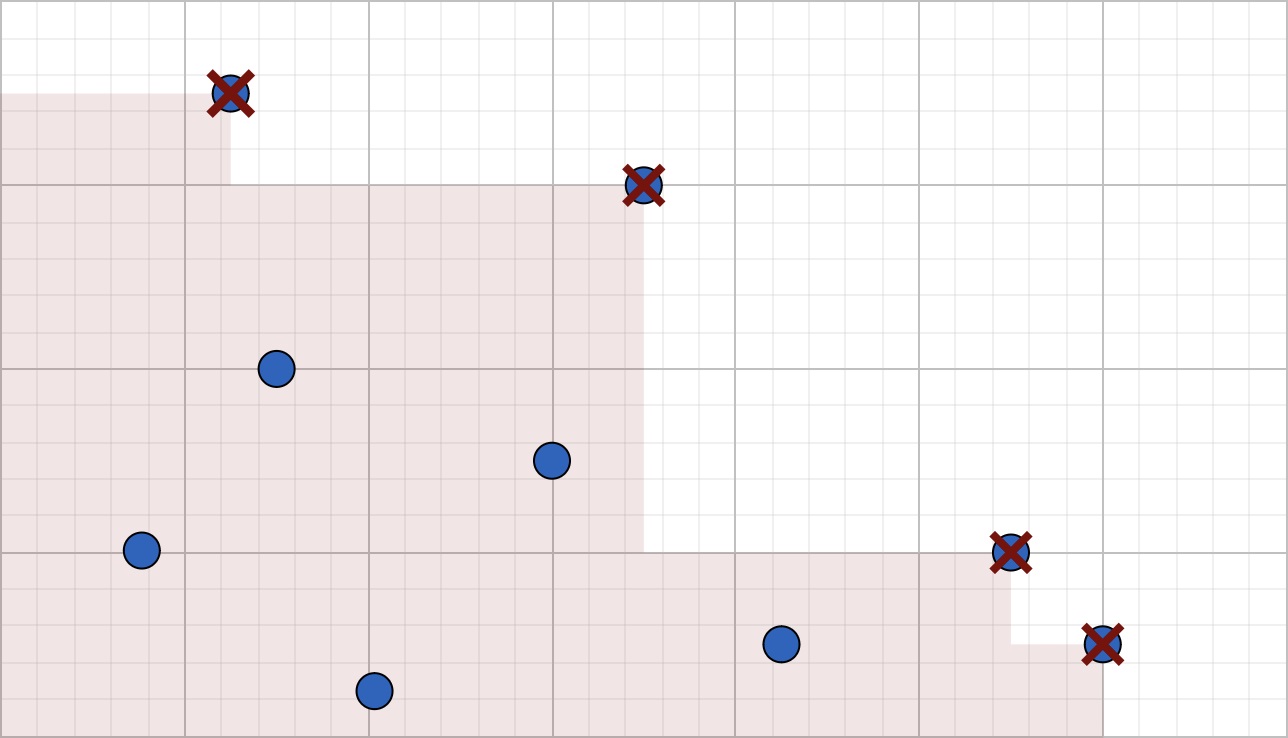}
}
\caption{\label{fig:0} The blue circles represent a set of vectors $A\subseteq \bR^2$ over which we take supremum with respect to coordinate-wise order.  
	Left: Red square depicts the \textit{vector} smallest upper bound for the collection $A$. Right: Red crossed circles represent the maximal (non-dominated) elements, while the shaded area is the lower set they generate that is a \textit{set} smallest upper bound for $A$.  }
\end{figure}

In this study, our primary focus is on the first approach---a vector-valued, ideal point, supremum. We chose this approach due to its tractability, leading to a vector or multi-objective optimization problem, rather than set-valued optimization problems. The second alternative, where the robust problem results in a set optimization problem, is a challenge that we defer to future work. We discuss further connections to the robust multi-objective literature in Section~\ref{sec:component-wise}.

\begin{definition}\label{def:supremum0}
Given a finite collection of (random) vectors $\{X^\theta\}_{\theta\in \Theta} \subseteq \sL_t(\bR^d)$, \textit{a supremum} with respect to the preorder  $\preceq^t$ is a (random) vector $V\in \sL_t(\bR^d)$ satisfying:
\begin{enumerate}[(i)]
\item  $X^\theta \preceq^t V$, for all $\theta \in \Theta$;
\item If for some $A \in \sL_t(\bR^d)$ we have $X^\theta \preceq^t A$ for all $\theta \in \Theta$, then  $V \preceq^t A$. 
\end{enumerate}
The set of all such suprema will be denoted by ${\Vsup\limits_{\theta\in\Theta}}^t X^\theta.$
\end{definition}

We note that in contrast to a supremum of a subset of $\bR$, a supremum in the sense of Definition~\ref{def:supremum0} may not exist, and when it exists it may be not unique. 
See Example~\ref{ex:supremum} below that illustrates this. Later in this section, we provide sufficient conditions for the existence of a suprema as well as for its uniqueness.

In what follows, we will denote by ${\vsup\limits_{\theta \in \Theta}}^t X^\theta $ one arbitrary chosen element from the set ${\Vsup\limits_{\theta \in \Theta}}^t X^\theta $. Unless otherwise stated, the obtained results do not depend on the choice the suprema, see Lemma~\ref{lemma:invariantSup}. 
If the time instance $t$ is clear, then we may simply write $\vsup_{\theta\in \Theta}$ and  $\Vsup_{\theta\in \Theta}$.

The two properties in the Definition~\ref{def:supremum0} correspond to the usual requirement that the \textit{supremum} is (i) an \textit{upper bound} and (ii) the \textit{smallest} among all upper bounds in $\sL_t(\bR^d)$. 
Conditions (i) and (ii) can be equivalently stated in terms of the ordering cone $\sL_t(C)$, respectively, as follows:
\begin{enumerate}[(i)]
\item 
$V \in \bigcap\limits_{\theta \in \Theta} (X^\theta + \sL_t(C)),$\
\item 
$\bigcap\limits_{\theta \in \Theta} (X^\theta + \sL_t(C)) \subseteq V + \sL_t(C)$.
\end{enumerate}

\begin{lemma}\label{lemma:invariantSup}
	For any $V,W\in {\Vsup\limits_{\theta \in \Theta}}^t X^\theta$, we have that 
	\begin{enumerate}[(a)]
		\item $V\preceq^t W$ and $W\preceq^t V$;
		\item $V\pm \sL_t(C) = W\pm \sL_t(C)$, and thus  $V\pm \sL_t(C) =  {\Vsup\limits_{\theta \in \Theta}}^t X^\theta \pm\sL_t(C)$;
		\item If $\preceq^t$ is a vector partial order, then $\vsup_\theta X^\theta$ with respect to $\preceq^t$ is unique $\bP$-a.s, if it exists.
		\item Assume that $\Vsup\limits_{\theta \in \Theta} X^\theta \neq \emptyset$. Then for all $b \in \bR^d$ it holds $\Vsup\limits_{\theta \in \Theta} (X^\theta +b) = \left(\Vsup\limits_{\theta \in \Theta} X^\theta \right) + b$. 
	\end{enumerate}
\end{lemma}

\begin{proof} (a) follows from the definition of $\vsup$. By (a) and definition   of the preorder we get $V+ \sL_t(C) = W +  \sL_t(C)$. Consequently  $V -\sL_t(C) = W -  \sL_t(C)$. 
If $\preceq^t$ is a vector partial order, then $\preceq^t$ is antisymmetric which implies (c). Finally, (d) follows from compatibility of the preorder with the vector structure.
\end{proof}

Similarly, one could consider the notion of \textit{ideal point} infimum defined as  $\vinf_{\theta\in\Theta}X^\theta := - \vsup_{\theta\in \Theta}(-X^\theta)$. However, in this work we use the \textit{ideal point} supremum only for the robust part (inner sup in \eqref{eq:rob_multi}), while we treat the outer inf in \eqref{eq:rob_multi} as minimization of a vector-valued objective function in the standard sense of vector optimization, see Section~\ref{sec:VOP}.

\begin{example}\label{ex:supremum}

i) To illustrate that an ideal point supremum \textit{may not exist}, consider the partial order generated by the convex, pointed, polyhedral ordering cone
$$
C = \text{cone } \left\lbrace \begin{pmatrix}1 \\ 0 \\ 1\end{pmatrix}, \begin{pmatrix}0 \\ 1 \\ 1\end{pmatrix}, \begin{pmatrix}-1 \\ 0 \\ 1\end{pmatrix}, \begin{pmatrix}0 \\ -1 \\ 1\end{pmatrix}, \begin{pmatrix}0.75 \\ 0.75 \\ 1\end{pmatrix},  \begin{pmatrix}-0.75 \\ 0.75 \\ 1\end{pmatrix},  \begin{pmatrix}-0.75 \\ -0.75 \\ 1\end{pmatrix},  \begin{pmatrix}0.75 \\ -0.75 \\ 1\end{pmatrix} \right\rbrace
$$
alongside two points $X = (0,0,0)^T$ and $Y = (0.25, -0.25, 0)^T$. It can be verified directly that the intersection $(X+C)\cap(Y+C)$ is a polyhedron with three vertices rather than a shifted cone; see Figure~\ref{fig:1} for a graphical illustration. Therefore, a supremum of vectors $X$ and $Y$ with respect to $\preceq_C$ does not exist.
\begin{figure}[h]
\centering
\subfloat{
\centering
\includegraphics[width = 0.4\textwidth]{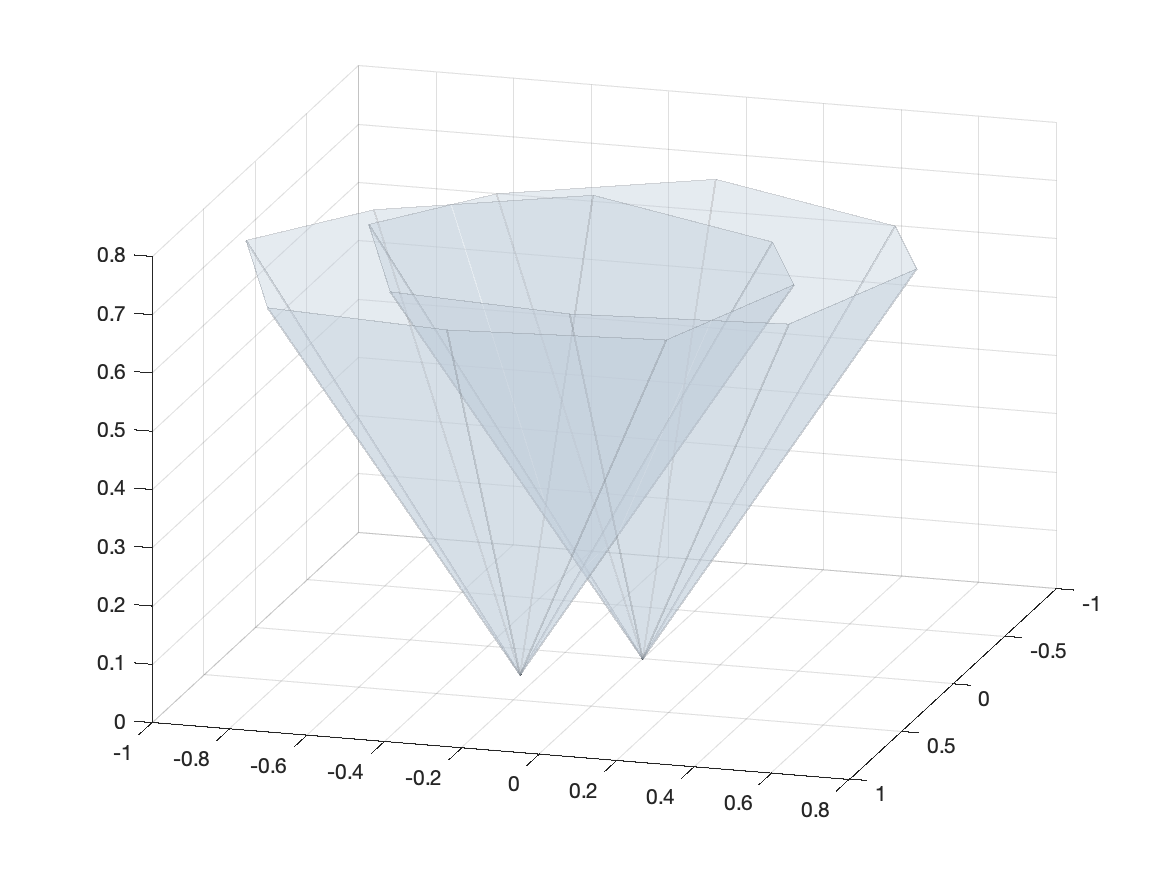}
}
~
\subfloat{
\includegraphics[width = 0.4\textwidth]{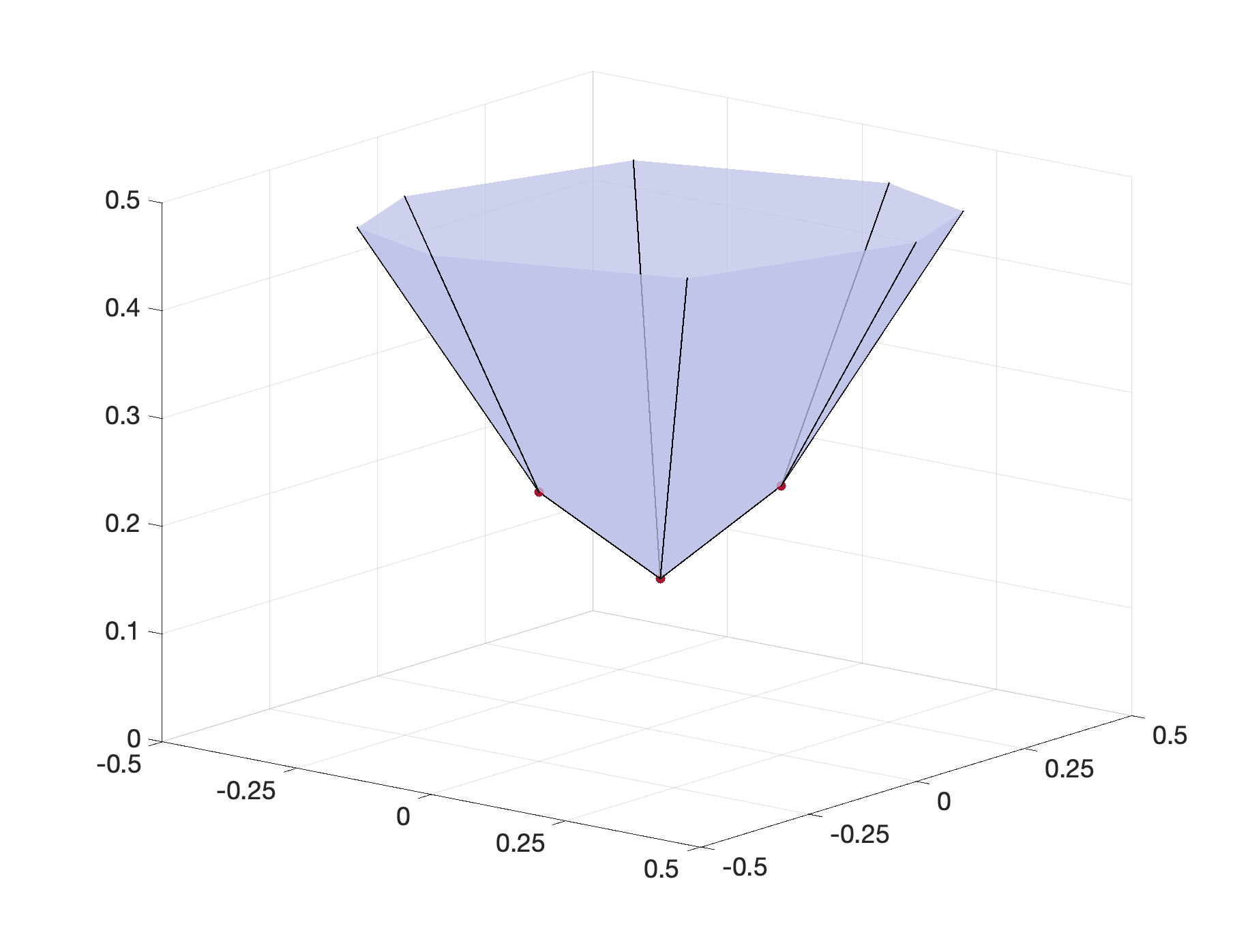}}
\caption{\label{fig:1} Shifted cones $X + C$ and $Y + C$ (left) and their intersection $(X+C)\cap(Y+C)$ (right).}
\end{figure}

\smallskip \noindent
ii) To illustrate that the ideal point supremum \textit{may not be unique}, consider a preorder generated by a (convex and polyhedral) ordering cone that is a half-space,
$$
C = \left\lbrace x \in \bR^d \; : \; \inprod{w}{x} \geq 0  \right\rbrace
$$
for some $w \in \bR^d \setminus \{0\}$. Then for any (finite) collection $\{X^\theta\}_{\theta \in \Theta} \subseteq \bR^d$   all points on the hyperplane
$$\left\lbrace x \in \bR^d \; : \; \inprod{w}{x} = \max \{\inprod{w}{X^\theta } : \theta \in \Theta\} \right\rbrace$$
are suprema of $\{X^\theta\}_{\theta \in \Theta}$ with respect to the preorder generated by $C$.

\end{example}

In the context of the existence of the supremum, let us mention two relevant cases, where the existence of the ideal point supremum is guaranteed. First, if we take the ordering cone $C = \bR^d_+$, which corresponds to component-wise partial order $\leq$, then the supremum exists, is unique and can be computed explicitly in a component-wise fashion, see Section~\ref{sec:component-wise} for more details. Second, for $d=2$ objectives, solid convex cones in $\bR^2$ are polyhedral and existence of suprema follows from Lemma~\ref{lemma:cone2Sufficeint} below. 

\begin{lemma} \label{lemma:cone1}
	Assume that the ordering cone $C \subseteq \bR^d$ is solid. Then for any $X,Y \in \sL_t(\bR^d)$ it holds $(X + \sL_t(C)) \cap (Y + \sL_t(C)) \neq \emptyset$.
\end{lemma}
\begin{proof}
	Let $\norm{\cdot}$ be a norm on $\sL_t(\bR^d)$ and $B_1 := \{ X \in \sL_t(\bR^d) : \norm{X} \leq 1\}$ be the closed unit ball in that norm.  Since $C$ is solid, then there exists $Z \in \sL_t(\Omega;\text{Int } C)$, Without loss of generality, we assume that  $Z$ is scaled in such a way that $Z + B_1 \subseteq \sL_t(C)$. 
	Then $V := Y + \norm{X -Y} \cdot Z = X + \norm{X -Y} \cdot \left( Z + \frac{Y - X}{\norm{X - Y}} \right)\in (X + \sL_t(C)) \cap (Y + \sL_t(C))$, since $Y + \norm{X -Y} \cdot Z \in Y + \sL_t(C)$ and $X + \norm{X -Y} \cdot \left( Z + \frac{Y - X}{\norm{X - Y}} \right) \in X + \text{ cone}(Z + B_1) \subseteq X + \sL_t(C)$. The proof is complete. 
\end{proof}

We recall that a cone $C$ is non-trivial, if $\{0\} \subsetneq C \subsetneq \bR^d$, and it is polyhedral if it is equal to the intersection of finite number of halfspaces, where a half space is a set of the form $\{ z \in\bR^d : \inprod{a}{z} \geq \alpha \}$ for some $a \in \bR^d \setminus\{0\}$ and $\alpha \in \bR$. 

\begin{lemma}\label{lemma:cone2Sufficeint}
Assume that the ordering cone $C\subseteq \bR^d$ is solid, convex and non-trivial . Additionally, suppose that $C$ is a polyhedral such that the dual cone $C^+:= \{ b \in \bR^d : \inprod{b}{c} \geq 0, \; \forall c \in C\}$ is generated by some linearly independent vectors $b_1, \dots, b_k \in \bR^d \setminus \{0\}$.
Then for any  $X,Y\in \sL_t(\bR^d)$, there exists $V\in \sL_t(C)$ such that $(X+\sL_t(C)) \cap (Y+ \sL_t(C)) = V + \sL_t(C).$ 
\end{lemma}
\begin{proof}
For $a \in \bR^d \setminus\{0\}$ and $\alpha \in \sL_t(\bR)$ we denote the halfspace by $H^+_\alpha (a) = \{ Z \in \sL_t(\bR^d) : \inprod{a}{Z} \geq \alpha, \; \bP\text{-a.s.} \}$ and the corresponding `hyperplane' by  $H^=_\alpha (a) = \{ Z \in \sL_t(\bR^d) : \inprod{a}{Z} = \alpha, \; \bP\text{-a.s.} \}$ for sets of random vectors.

In view of the assumptions, the polyhedral convex cone $C$ can be written as 
$$
C = \bigcap\limits_{i = 1, \dots, k} H^+_0(b_i).
$$
Let $0 \in \sL_t(\bR)$ denote a random variable constantly equal to zero. Then it also holds 
$$\sL_t(C) = \bigcap\limits_{i = 1, \dots, k} H^+_{0} (b_i) = \{ Z \in \sL_t(\bR^d) \; : \; \inprod{b_i}{Z} \geq 0, \bP\text{-a.s.}, \;  i = 1, \dots, k  \}.$$ 
First, we show that $X + \sL_t(C) = \bigcap\limits_{i = 1, \dots, k} H^+_{\inprod{b_i}{X}}(b_i)$. Take arbitrary $\bar{c} \in \sL_t(C)$, for $i = 1, \dots, k$ it holds $\inprod{b_i}{X + \bar{c}} \geq \inprod{b_i}{X}$, which shows $X + \sL_t(C) \subseteq \bigcap\limits_{i = 1, \dots, k} H^+_{\inprod{b_i}{X}}(b_i)$. Now take $Z \in \bigcap\limits_{i = 1, \dots, k} H^+_{\inprod{b_i}{X}}(b_i)$, then for $i = 1, \dots, k$ it holds $\inprod{b_i}{Z - X} \geq 0$, so $Z-X \in \sL_t(C)$ and $X + \sL_t(C) \supseteq \bigcap\limits_{i = 1, \dots, k} H^+_{\inprod{b_i}{X}}(b_i)$. Equality $Y + \sL_t(C) = \bigcap\limits_{i = 1, \dots, k} H^+_{\inprod{b_i}{Y}}(b_i)$ follows analogously.

Now consider $\omega \in \Omega$ and for $i = 1, \dots, k$ set $\alpha_i(\omega) := \max \{ \inprod{b_i}{X(\omega)}, \inprod{b_i}{Y(\omega)} \}$.
Note that $\alpha_i\in\sL_t(\bR)$. 
  Then it holds 
\begin{equation}\label{eq:cone2-1}
	(X + \sL_t(C)) \cap (Y + \sL_t(C)) = \bigcap\limits_{i = 1, \dots, k} H^+_{\alpha_i}(b_i).
\end{equation}	
	
In view of Lemma~\ref{lemma:cone1}, the intersection in the right hand side of \eqref{eq:cone2-1} is non-empty. 
Next, let us consider the intersection of the corresponding hyperplanes,
\begin{align*}
\bigcap\limits_{i = 1, \dots, k} H^=_{\alpha_i}(b_i) = 
\left\lbrace Z \in \sL_t(\bR^d) \quad\vert\quad \inprod{b_i}{Z(\omega)} = \alpha_i(\omega), \bP\text{-a.s.}, \quad  i = 1, \dots, k  \right\rbrace.
\end{align*}
Since $b_1, \dots, b_k$ are linearly independent, the system of linear equations $\inprod{b_i}{Z(\omega)} = \alpha_i(\omega), \quad  i = 1, \dots, k$ has an $\sF_t$-measurable solution for arbitrary $\sF_t$-measurable right-hand side $(\alpha_1, \dots, \alpha_k)$. Therefore, $\bigcap\limits_{i = 1, \dots, k} H^=_{\alpha_i}(b_i) \neq \emptyset$, $\bP$-a.s..

Finally, we show that for arbitrary $ V \in \bigcap\limits_{i = 1, \dots, k} H^=_{\alpha_i}(b_i)$ we have $\bigcap\limits_{i = 1, \dots, k} H^+_{\alpha_i}(b_i) =  V + \sL_t(C)$. Take any $Z \in \bigcap\limits_{i = 1, \dots, k} H^+_{\alpha_i}(b_i)$, then for all $i = 1, \dots, k$ it holds $\inprod{b_i}{Z-V} \geq 0$, so $Z-V \in \sL_t(C)$ and $\bigcap\limits_{i = 1, \dots, k} H^+_{\alpha_i}(b_i) \subseteq  V + \sL_t(C)$. On the other hand, for any $\bar{c} \in \sL_t(C)$ we have for all $i=1, \dots, k$ that $\inprod{b_i}{V + \bar{c}} \geq \inprod{b_i}{V} = \alpha_i$, so $\bigcap\limits_{i = 1, \dots, k} H^+_{\alpha_i}(b_i) \supseteq  V + \sL_t(C)$.

\end{proof}

Lemma~\ref{lemma:cone2Sufficeint} provides sufficient conditions for existence of an \textit{ideal point} supremum. Indeed, the result extends by induction from two random vectors to a finite collection of random vectors and hence under the assumptions of Lemma~\ref{lemma:cone2Sufficeint} for a finite collection $\{X^\theta\}_{\theta \in \Theta} \subseteq \sL_t(\bR^d)$ there exists a random vector $V \in \sL_t(\bR^d)$ satisfying
$$\bigcap\limits_{\theta \in \Theta} (X^\theta + \sL_t(C)) \subseteq V + \sL_t(C),$$
which means that $V$ is a supremum of $\{X^\theta\}_{\theta \in \Theta}$ with respect to $\preceq^t$ in the sense of Definition~\ref{def:supremum0}.

\begin{remark}\label{rem:2d-sup}
In the bi-objective case $d=2$, each solid, closed, convex cone $C \subseteq \bR^2$ is either a half-space or an intersection of two half-spaces. Therefore, by Lemma~\ref{lemma:cone2Sufficeint} an ideal point supremum in this case always exists. Moreover, for any pair $x, y \in \bR^2$ we can explicitly compute a vector $z \in \bR^2$ such that $(x+C) \cap (y+C) = z+C$ through a solution of two linear equations. By analogous  arguments, $\vsup$ exists and can be explicitly computed for any finite collection of two-dimensional random vectors. 
\end{remark}

The remainder of this section provides additional properties of $\vsup$ needed in the sequel, such as monotonicity, uniqueness and equivalent characterizations.

\begin{lemma}\label{lemma:monotone}
Assume that $\{X^\theta \}_{\theta \in \Theta}, \{Y^\theta \}_{\theta \in \Theta} \subseteq \sL_t(\bR^d)$, and $X^\theta \preceq^t Y^\theta$ for all $\theta \in \Theta$, then $\vsup\limits_{\theta \in \Theta} X^\theta \preceq^t \vsup\limits_{\theta \in \Theta} Y^\theta$, assuming that both suprema exist. 
\end{lemma}
\begin{proof}
For any $\theta \in \Theta$, by transitivity $X^\theta \preceq^t Y^\theta$ and $Y^\theta \preceq^t \vsup\limits_{\theta \in \Theta}Y^\theta$ imply $X^\theta \preceq^t \vsup\limits_{\theta \in \Theta}Y^\theta$. Thus,  $\vsup\limits_{\theta \in \Theta}Y^\theta$ is an upper bound of $\{X^\theta \}_{\theta \in \Theta} \subseteq \sL_t(\bR^d)$. 
The result follows from Definition~\ref{def:supremum0}(ii).
\end{proof}

\begin{lemma}\label{lemma:supt+1}
	Let $\set{X^\theta}\subseteq \sL_{t+1}(\bR^d)$, and assume that $V,W\in{\Vsup\limits_{\theta\in \Theta}}^{t+1} X^\theta$. Then 
	\[
	{\Vsup_{\theta\in \Theta}}^t \bE^\theta_t[V] + \sL_t(C)=	{\Vsup_{\theta\in \Theta}}^t \bE^\theta_t[W] + \sL_t(C), 
	\]
		assuming both sets are not empty. 
\end{lemma}
\begin{proof}
By Lemma~\ref{lemma:invariantSup}(a), Lemma \ref{lemma:order_exp} and Lemma~\ref{lemma:monotone},then 
\[
{\vsup_{\theta\in\Theta}}^t\bE^\theta_t[V] \preceq^t{\vsup_{\theta\in\Theta}}^t\bE^\theta_t[W], \quad \textrm{and} \quad {\vsup_{\theta\in\Theta}}^t\bE^\theta_t[W] \preceq^t {\vsup_{\theta\in\Theta}}^t\bE^\theta_t[V],
\]
and note that all elements exist. Consequently, by the definition of $\preceq^t$, we get 
\[
{\vsup_{\theta\in\Theta}}^t\bE^\theta_t[V] + \sL_t(C) = {\vsup_{\theta\in\Theta}}^t\bE^\theta_t[W] +\sL_t(C),
\]
and in view of Lemma \ref{lemma:invariantSup}(b) the proof is complete. 
\end{proof}

\subsection{Vector optimization problem}\label{sec:VOP}
In this section, we discuss how to interpret the optimization with respect to feasible controls $\varphi\in\mathfrak{A}$ (the infimum in \eqref{eq:rob_multi}) in the context of our multi-objective robust stochastic control problem. Since our choice for supremum across models was a vector, the infimum across controls will be viewed as a vector optimization problem (VOP). 

A VOP is a problem of the form
\begin{align}
	\label{VOP}
	\tag{VOP}
	\text{minimize } \; F(x) \quad \text{ with respect to } \; \preceq \quad \text{ subject to } \; x \in \bX,
\end{align}
where $\bX \subseteq \bR^n$ is a nonempty feasible set, $F: \bX \to \bR^d$ is a vector-valued mapping, and $\preceq$ is a fixed preorder on $\bR^d$ with the corresponding ordering cone $C$.  
The \textit{image of the feasible set} $\bX$ is the set $F[\bX] = \{ F(x) : x \in \bX\}$ and the \textit{upper image} of~\eqref{VOP} is $\cP = \text{cl } (F[\bX] + C)$, where $ \text{cl }(A)$ denotes the closure of set $A\subseteq \bR^d$ with respect to a fixed topology in $\bR^d$. These two sets, the image of the feasible set and the upper image, play important roles within vector and set optimization theory. Next, we shortly outline their connection to the notions of Pareto optimal or efficient points and to a notion of solution of a vector optimization problem. 

A point $\bar{x} \in \bX$ is called a \textit{minimizer} for \eqref{VOP} if  $(F(\bar{x}) - C \setminus \{0\}) \cap F[\bX] = \emptyset$.  In various streams of literature, a minimizer is also referred to as a \textit{Pareto optimal point} or an \textit{efficient point}.  The set of all images of minimizers of \eqref{VOP} is often known as the \textit{Pareto frontier} or the \textit{efficient frontier}. There exists also the notion of a weak minimizer.

In numerous vector optimization problems the Pareto frontier is non-trivial containing mutually incomparable points with respect to considered preorder. This highlights the difficulties in defining what a \textit{solution} of~\eqref{VOP} should be. In certain contexts,  solving a VOP can be understood as finding one Pareto optimal point, while in other contexts, one aims to identify (or approximate) the entire Pareto frontier or the upper image.  
Consequently, it does not come as a surprise that the current literature contains several concepts of solutions for VOPs based on case by case relevant properties (e.g.~bounded versus unbounded, convex versus non-convex) of the VOPs themselves.  We refer the reader to \cite{Loehne2011, Jahn04}, and references therein, for more details.

In this work, our main objective is to establish a dynamic programming principle for vector-valued robust stochastic problems. As such, we will be interested in an appropriately defined `value function', for which, as it turns out, the relevant objects are  the image of the feasible set and the upper image, while the particular notion of the solutions of VOP is less important.

\section{Dynamic programming under a preorder}
\label{sec:dyn-prog-preorder}
In what follows, we make the following standing assumption on the preorder $\preceq$:
\begin{assumption}\label{assum:vsup}
For all $t \in \cT$ and all finite collections of random vectors $\{X^\theta\}_{\theta \in \Theta} \subseteq \sL_t(\bR^d)$, there exists a supremum $\vsup$ compatible with $\preceq^t$.   
\end{assumption}

Now, we can formulate time $t$ robust vector optimization problem
\begin{align}
\begin{split}
\label{prob_robust}
\textrm{minimize }& \quad \vsup\limits_{\theta \in \Theta}  \bE^\theta_t [\ell(S^\varphi_T)], \quad \textrm{ with respect to } \preceq^t \\
\textrm{subject to: }& \quad \varphi \in \mathfrak{A}^t(S_t),
\end{split}
\end{align}
where, we recall that $\mathfrak{A}^t(S_t)$ is the set of admissible controls $\varphi=(\varphi_t,\ldots, \varphi_{T-1})$ starting at time $t$, and given value $S_t$. 
Given the \textit{ideal point} (vector-valued) understanding of a supremum, problem~\eqref{prob_robust} is a well-defined (stochastic) vector optimization problem.

We are interested in dynamic programming principle and time consistency property of the family of these robust problems.
For this purpose, we will use the image of the feasible set as the set-valued candidate for the value function, 
\begin{align*}
\cV_t(S_t) &:=  \left\lbrace \vsup\limits_{\theta \in \Theta}  \bE^\theta_t [\ell(S_T)] \ : \ \varphi \in \mathfrak{A}^t(S_t), S_{s+1} = F(s, S_s, \varphi_s, Z_{s+1}), s = t, \dots, T-1  \right\rbrace   \\
& =  \bigcup\limits_{\varphi \in \mathfrak{A}^t(S_t)} \vsup\limits_{\theta \in \Theta}  \bE^\theta_t [\ell(S^\varphi_T)], \quad t=0,1,\ldots, T-1, 
\end{align*}
and we put $\cV_T(S_T)= \ell(S_T)$. Note that the value function consists of suprema over all time $t$ feasible strategies.
We recall that if $\vsup$ is not unique, we take one of them. We will show that our results do not depend on the choice of the $\vsup$. 

\begin{remark}
\label{rem:upper-im}
 Dynamic programming in the context of non-standard problems (without model uncertainty)  has been explored in the literature   for  multi-objective stochastic control problems in~\cite{KovacovaRudloff2019}, Nash-equilibria of non-zero sum games in~\cite{FeinsteinEtAl2022} or computation of multivariate risk measures in~\cite{FeinsteinRudloff2017}. In these works the version of the Bellman's principle is based on a set-valued notion of a value function -- such as the image of the feasible set $\cV_t (\cdot)$ defined herein. An alternative to this value function is (a set-valued value function attaining as values) the upper image of the multi-objective problem. We refer to~\cite[Section 5.1]{KovacovaRudloff2019} for an interpretation of an upper image as a (set-optimization) infimum. We also discuss the upper image approach for a particular case of component-wise partial order; see Section~\ref{sec:component-wise}. 
\end{remark}

\subsection{Weak Bellman's principle}
In parallel to the scalar case,  we aim to derive Bellman-type backward recursions for the value function $\cV$. For this purpose, we introduce a one-step recursively constructed sets

$$
\cR_t (S_t, \cV_{t+1}) := \left\lbrace {\vsup\limits_{\theta \in \Theta}}^t  \bE^\theta_t [X] \quad \vert \quad \varphi_t \in \mathfrak{A}_t(S_t), X \in \cV_{t+1}( F(t,S_t,\varphi_t, Z_{t+1})) \right\rbrace.
$$
where $\mathfrak{A}_t(S_t)\subset A$ is the set of admissible one step controls at time $t$ given $S_t$.

Note that, in the spirit of dynamic programming, $\cR_t(\cdot, \cV_{t+1})$ can be viewed as the value function of a one-step vector optimization problem
	\begin{align}
		\begin{split}
			\label{prob:rec1}
			\textrm{minimize }& \quad {\vsup\limits_{\theta \in \Theta}}^t  \bE^\theta_t [X], \quad \textrm{ with respect to } \preceq^t \\
			\textrm{subject to: }& \quad \varphi_t \in \mathfrak{A}_t(S_t), \; X \in \cV_{t+1}( F(t,S_t,\varphi_t, Z_{t+1})).
		\end{split}
	\end{align}
However, a closer inspection of $\cR_t$ shows that the one-step problem assumes knowledge of value function $\cV_{t+1}$, and it is not truly (backward) recursive. The following theorem illustrates order relation(s) between the value function $\cV$ and the one-step recursive $\cR$. This is the first step towards a Bellman-type relation(s) for the robust problem.

\begin{theorem}\label{th:Bellman1}
For every time $t = 0, 1, \dots, T-1$ 
we have
\begin{align}\label{eq:valueF1}
\cV_t(S_t) \preccurlyeq^t \cR_t (S_t, \cV_{t+1}), \ \text{i.e. } \cR_t (S_t, \cV_{t+1}) \subseteq \cV_t(S_t) + \sL_t(C), 
\end{align}
as well as
\begin{align}\label{eq:valueF2}
\cV_t(S_t) \curlyeqprec^t \cR_t (S_t, \cV_{t+1}), \ \text{i.e. }  \cV_t(S_t) \subseteq \cR_t (S_t, \cV_{t+1}) - \sL_t(C).
\end{align}
\end{theorem}
\begin{proof}
We start by noting that 
\begin{equation}\label{eq:th22-1}
\cV_t(S_t) + \sL_t(C) = \bigcup\limits_{\varphi \in \mathfrak{A}^t(S_t)} \vsup\limits_{\theta \in \Theta}  \bE^\theta_t [\ell(S^\varphi_T)] + \sL_t(C) =
\bigcup\limits_{\varphi \in \mathfrak{A}^t(S_t)} \left(\Vsup\limits_{\theta \in \Theta}  \bE^\theta_t [\ell(S^\varphi_T)] +  \sL_t(C)\right),
\end{equation}	
and in view of Lemma \ref{lemma:invariantSup}(b), the right hand side does not depend on choice of $\vsup$.
	
To prove \eqref{eq:valueF1}, take arbitrary $\varphi_t \in \mathfrak{A}_t(S_t)$ and $X \in \cV_{t+1}(S_{t+1}^{\varphi_t})$, where $S_{t+1}^{\varphi_t} = F(t, S_t, \varphi_t, Z_{t+1})$. For $X \in \cV_{t+1}(S^{\varphi_t}_{t+1})$ there must exists $\varphi^X \in \mathfrak{A}^{t+1}(S_{t+1})$. 
Thus, $\bar{\varphi} := (\varphi_t, \varphi^X_{t+1}, \ldots, \varphi^X_{T-1}) \in \mathfrak{A}^t(S_t)$ generates the terminal state $S^{\bar{\varphi}}_T = S^{\varphi^X}_T$. By definition of supremum, since $X\in{\Vsup\limits_{\theta \in \Theta}}^{t+1}  \bE^\theta_{t+1} [\ell(S^{\bar{\varphi}}_T)]$,  we have
\begin{align*}
\bE^\theta_{t+1} \left[ \ell (S^{\bar{\varphi}}_T) \right] \preceq^{t+1} 
=  X, \quad \forall \theta \in \Theta.
\end{align*}
By Lemma~\ref{lemma:order_exp} we also have
\begin{align*}
\bE^\theta_t \left[  \ell (S^{\bar{\varphi}}_T)  \right] = \bE^\theta_t \left[ \bE^\theta_{t+1} \left[ \ell (S^{\bar{\varphi}}_T) \right] \right] \preceq^t \bE^\theta_t \left[ X \right], \quad \forall \theta \in \Theta.
\end{align*}
By monotonicity of the supremum operator, Lemma~\ref{lemma:monotone} we have
\begin{align*}
\vsup\limits_{\theta \in \Theta} \bE^\theta_t \left[  \ell (S^{\bar{\varphi}}_T)  \right] \preceq^t \vsup\limits_{\theta \in \Theta} \bE^\theta_t \left[ X \right].
\end{align*}
Thus,  using Lemma \ref{lemma:invariantSup} and \eqref{eq:th22-1}, we obtain
\[
\vsup\limits_{\theta \in \Theta} \bE^\theta_t \left[ X \right] \in 
\vsup\limits_{\theta \in \Theta} \bE^\theta_t \left[  \ell (S^{\bar{\varphi}}_T)  \right] 
+ \sL_t(C) \subseteq \bigcup\limits_{\varphi \in \mathfrak{A}^t(S_t)} \left(\vsup\limits_{\theta \in \Theta}  \bE^\theta_t [\ell(S^\varphi_T)] + \sL_t(C)\right) = \cV_t(S_t) + \sL_t(C).  
\]
Hence,  \eqref{eq:valueF1} is proved. 

To prove \eqref{eq:valueF2}, let $V\in \cV_t(S_t)$. Then, $V={\vsup\limits_{\theta\in \Theta}}^t \bE_t^\theta[\ell (S_T^\varphi)]$, for some $\varphi = (\varphi_t,\tilde\varphi)\in \mathfrak{A}^t(S_t)$, where $\tilde\varphi = (\varphi_{t+1}, \ldots, \varphi_{T-1})\in\mathfrak{A}^{t+1} (S_{t+1}^{\varphi_t})$, with assumed notation $S_{t+1}^{\varphi_t}:=F(t,S_t,\varphi_t,Z_{t+1})$.  Note that $S_T^{\varphi} = S_T^{\tilde\varphi}$

By definition of supremum, it holds
\begin{align*}
\bE^\theta_{t+1} [\ell(S^\varphi_T)] \preceq^{t+1} {\vsup\limits_{\theta \in \Theta}}^{t+1} \bE^\theta_{t+1} [\ell (S^{\tilde\varphi}_T)], \quad \forall \theta \in \Theta.
\end{align*}
By Lemma~\ref{lemma:order_exp},  it follows that
\begin{align*}
\bE^\theta_t \left[ \ell (S^\varphi_T) \right] = \bE^\theta_t \left[ \bE^\theta_{t+1} [\ell(S^\varphi_T)] \right] 
\preceq^{t} \bE^\theta_t \left[{\vsup\limits_{\theta' \in \Theta}}^{t+1} \bE^{\theta'}_{t+1} [\ell(S^{\tilde\varphi}_T)] \right], \quad \forall \theta \in \Theta.
\end{align*}
By monotonicity of the supremum operator we have 
\begin{align}
V={\vsup\limits_{\theta \in \Theta}}^t \bE^\theta_t \left[ \ell(S^\varphi_T) \right] 
\preceq^{t} {\vsup\limits_{\theta \in \Theta}}^t \bE^\theta_t \left[ {\vsup\limits_{\theta' \in \Theta}}^{t+1} \bE^{\theta'}_{t+1} [\ell(S^{\tilde\varphi}_T)] \right]=: R. \label{eq:Rvalue}
\end{align}
By construction $R\in \cR_t(S_t,\cV_{t+1})$, and thus $V\in R- \sL_t(C)$. In view of Lemma~\ref{lemma:invariantSup}(b), the set $R- \sL_t(C)$ is the same regardless of the choice of $R$ as $\vsup^t$ in the right hand side of \eqref{eq:Rvalue}.
Therefore, $\cV_t(S_t) \subseteq \cR_t (S_t, \cV_{t+1}) - \sL_t(C)$, and the proof is complete. 
\end{proof}

We remark that the sets $\cR_t$ was built by a forward procedure, while deriving a dynamic programming principle we should aim for backward recursions. For this reason, we consider the sets
\begin{align*}
	\cB_{T} (S_{T}) &:= \ell(S_{T}), \\
	\cB_{t} (S_{t}) &:= \left\lbrace \vsup\limits_{\theta \in \Theta}  \bE^\theta_t [X] \mid  \varphi_t \in \mathfrak{A}_t(S_t), X \in \cB_{t+1}(F(t, S_t,\varphi_t, Z_{t+1})) \right\rbrace,  \ \text{ for } t = T-1, \dots, 0.
\end{align*}
In this case, $\cB_t(\cdot)$ also corresponds to value function of a one step optimization problem
	\begin{align}
		\begin{split}
			\label{prob:rec2}
			\textrm{minimize }& \quad {\vsup\limits_{\theta \in \Theta}}^t  \bE^\theta_t [X], \quad \textrm{ with respect to } \preceq^t \\
			\textrm{subject to: }& \quad \varphi_t \in \mathfrak{A}_t(S_t), \; X \in \cB_{t+1}(F(t, S_t,\varphi_t, Z_{t+1})).
		\end{split}
	\end{align}
Note that in contrast to one-step problems~\eqref{prob:rec1},  problems~\eqref{prob:rec2} are genuinely backward recursive in the
sense of dynamic programming paradigm.

The question whether a Bellman-type relation(s) hold for the robust problem then is not a question of whether order relation(s) hold between $\cV$ and $\cR$, but whether order relation(s) hold between $\cV$ and $\cB$. The following results shows that such Bellman-type relation(s) hold true for the robust problem.

\begin{theorem}\label{th:Bellman0-1}
For every time $t = 0, 1, \dots, T-1$ and every state $S_t$  we have
\begin{align} \label{eq:proofVB1} 
\cV_t(S_t) \preccurlyeq^t \cB_t (S_t), \ \textrm{i.e} \    \ \cB_t (S_t) \subseteq \cV_t(S_t) + \sL_t(C),
\end{align}
as well as
\begin{align} \label{eq:proofVB2}
\cV_t(S_t) \curlyeqprec^t \cB_t (S_t), \ \textrm{i.e.} \   \ \cV_t(S_t) \subseteq \cB_t (S_t) - \sL_t(C).
\end{align}
\end{theorem}
\begin{proof}
We will prove the result by backward induction. 
By definitions  $\cV_T(S_T) = \ell(S_T) = \cB_T(S_T)$.
Next, we will prove \eqref{eq:proofVB1}. 
As an induction hypothesis assume that the relation $\cB_{t+1} (S_{t+1}) \subseteq \cV_{t+1}(S_{t+1}) + \sL_{t+1}(C)$, for any state $S_{t+1}$. 
Let $B\in  \cB_t (S_t)$. Then, $B\in {\Vsup\limits_{\theta \in \Theta}}^t \, \bE^\theta_t [X] $, for some $X\in\cB_{t+1}(S_{t+1}^{\varphi_t})$, $\varphi_t\in \mathfrak{A}_t(S_t)$. 

According to the induction hypothesis there exists $V \in \cV_{t+1}(S_{t+1}^{\varphi_t})$ such that
$V \preceq^{t+1} X$. Then, by Assumption~\ref{assum:vsup} and Lemmas~\ref{lemma:order_exp}, \ref{lemma:monotone} and \ref{lemma:invariantSup}(a) it follows
\begin{align*}
	{\vsup\limits_{\theta \in \Theta}}^{t} \bE^\theta_t [V] \preceq^t {\vsup\limits_{\theta \in \Theta}}^{t} \bE^\theta_t [X] \preceq^t B 
	\end{align*}
and thus 
\[
 B\in 	{\vsup\limits_{\theta \in \Theta}}^{t} \bE^\theta_t [V] +\sL_t(C). 
\]
Note that ${\vsup\limits_{\theta \in \Theta}}^{t} \bE^\theta_t [V]  +\sL_t(C) \subseteq \cR_t (S_t, \cV_{t+1})  +\sL_t(C)$, and by Theorem~\ref{th:Bellman1} 
${\vsup\limits_{\theta \in \Theta}}^{t} \bE^\theta_t [V] \in \cV_t(S_t)+\sL_t(C)$, which combined with the above inclusion yields \eqref{eq:proofVB1}.

Let us next prove \eqref{eq:proofVB2}. As an induction hypothesis assume that the $\cV_{t+1} (S_{t+1}) \subseteq \cB_{t+1}(S_{t+1}) - \sL_{t+1}(C)$, for any state $S_{t+1}$. Consider arbitrary $V \in \cV_t(S_t)$. 

Recall that by Theorem~\ref{th:Bellman1}, $\cV_t(S_t) \curlyeqprec^t \cR_t (S_t, \cV_{t+1})$. Therefore, for this given $V \in \cV_t(S_t)$, there exists $\varphi_t \in \mathfrak{A}_t(S_t)$ and $X \in \cV_{t+1}(S_{t+1}^{\varphi_t})$ such that
\begin{align}\label{eq:proofVB22}
V \preceq^t {\vsup\limits_{\theta \in \Theta}}^{t} \bE^\theta_t [X].
\end{align}
On the other hand, since $X \in \cV_{t+1}(S_{t+1}^{\varphi_t})$, by induction hypothesis it holds $\cV_{t+1}(S_{t+1}^{\varphi_t}) \curlyeqprec^{t+1} \cB_{t+1} (S_{t+1}^{\varphi_t})$, therefore there exists $B \in \cB_{t+1} (S_{t+1}^{\varphi_t})$ such that
\[
X \preceq^{t+1} B.
\]
From Assumption~\ref{assum:vsup} and Lemmas~\ref{lemma:order_exp} and~\ref{lemma:monotone} it follows
\begin{align}\label{eq:proofVB33}
{\vsup\limits_{\theta \in \Theta}}^{t} \bE^\theta_t [X] \preceq^{t} {\vsup\limits_{\theta \in \Theta}}^{t} \bE^\theta_t [B],
\end{align}
By definition of set $\cB_t(S_t)$, and Lemma~\ref{lemma:invariantSup}(b),  ${\vsup\limits_{\theta \in \Theta}}^{t} \bE^\theta_t [B]  \in  \cB_t (S_t) - \sL_t(C)$. From here, by transitivity of $\preceq^t$ and \eqref{eq:proofVB22} and \eqref{eq:proofVB33}, we have that $V\in\cB_t (S_t) - \sL_t(C)$. This concludes the proof. 
\end{proof}

In the standard (single-objective) setting, the famous Bellman equation structurally corresponds to a recursive relation of value function of the problem. In the above theorems, we derived order relation(s) between the value function $\cV$ of the robust problem~\eqref{prob_robust} and the recursively constructed $\cB$, which corresponds to the value function of a one-time-step version of the robust problem. Therefore, the relation(s) proven in Theorem~\ref{th:Bellman0-1} can be interpreted as Bellman-type inclusion(s) for the robust problem, that is weaker forms of Bellman equations. A natural next question is to ask whether also reverse order relation(s) hold between $\cV$ and $\cB$.

\subsection{Strong Bellman's principle} 
Now we additionally assume that our family of models $\Theta$ has a \textit{rectangularity} with respect to the order relation (and the corresponding vector-valued supremum) considered here.
\begin{definition}
	\label{def:rectangle}
	We say that the family of models $\Theta$ is \textit{$\preceq$-rectangular} if for all times $t=0, 1, \dots, T-1$ and all random vectors $X \in \sL_T(\bR^d)$ it holds
	\begin{align}\label{eq:rect-weak}
		{\vsup\limits_{\theta \in \Theta}}^t  \bE^\theta_t  \left[ {\vsup\limits_{\theta \in \Theta}}^{t+1} \bE^\theta_{t+1} [X]  \right] 
		\preceq^t {\vsup\limits_{\theta \in \Theta}}^t \bE^\theta_t  [X].
	\end{align}
\end{definition}

\begin{remark}\label{rem:vsup-vsup} 
(i) Definition~\ref{def:rectangle} is well-posed, in the sense that \eqref{eq:rect-weak} does not depend on the choices of $\vsup$'s. 
Indeed, in view of Lemma~\ref{lemma:supt+1}, \eqref{eq:rect-weak} is invariant with respect to the choice of ${\vsup}^{t+1}$. On the other hand, by Lemma~\ref{lemma:invariantSup}(a) and transitivity of the $\preceq^t$, we have that \eqref{eq:rect-weak} does not depend on the choices of ${\vsup}^t$.

(ii) By transitivity and Lemma~\ref{lemma:monotone}, $\preceq$-rectangularity given by~\eqref{eq:rect-weak} also implies the following nested form of this property  
\begin{align*}
{\vsup\limits_{\theta \in \Theta}}^t  \bE^\theta_t  \left[ {\vsup\limits_{\theta \in \Theta}}^{t+1} \bE^\theta_{t+1} \left[ \dots {\vsup\limits_{\theta \in \Theta}}^{T-1} \bE^\theta_{T-1} [X] \dots \right]  \right] 
		\preceq^t {\vsup\limits_{\theta \in \Theta}}^t \bE^\theta_t  [X].
\end{align*}
A corresponding nested formulation was used in \cite{Shapiro2016} to define rectangularity property in the context of one-dimensional random variables and standard suprema, see next section.

(iii) For a general preorder $\preceq$, constructing a $\preceq$-rectangular family of models as well as verifying $\preceq$-rectangularity property remain challenging problems. In \cite{Shapiro2016}, the author studies these questions for one-dimensional random variables, and provides a recursive construction of rectangular sets of probability measures. In Section~\ref{sec:component-wise} we show that a similar constructions implies $\leq$-rectangularity in $\bR^d$, which allows to explore the component-wise setting further.

(iv) Using properties of suprema, we have  
	$$
	{\vsup\limits_{\theta \in \Theta}}^t \bE^\theta_t  [X] 
	\preceq^t {\vsup\limits_{\theta \in \Theta}}^t  \bE^\theta_t  \left[ {\vsup\limits_{\theta \in \Theta}}^{t+1} \bE^\theta_{t+1} [X] \right]. 
	$$ 
	Hence, if the preorder $\preceq$ is partial order (i.e. also antisymmetric), then the induced preorders $\preceq^t$ are also antisymmetric and thus partial orders. Therefore, in case of a partial order the $\preceq$-rectangularity property takes the form
	\begin{align}\label{eq:rect-strong}
		{\vsup\limits_{\theta \in \Theta}}^t \bE^\theta_t  [X]  = {\vsup\limits_{\theta \in \Theta}}^t  \bE^\theta_t  \left[ {\vsup\limits_{\theta \in \Theta}}^{t+1} \bE^\theta_{t+1} [X]  \right].
	\end{align}

\hfill $\square$
\end{remark}

In the following two theorems we derive order relations between the three value functions, $\cV, \cR$ and $\cB$,  assuming $\preceq$-rectangularity. These will strengthen the derived DPP for the robust problem. 

\begin{theorem}\label{th:Bellman2}
	Assume that the family of models $\Theta$ is $\preceq$-rectangular. Then for every time $t = 0, 1, \dots, T-1$ and every state $S_t$  we have
	\begin{align*}
		\cR_t (S_t, \cV_{t+1}) \preccurlyeq^t \cV_t(S_t), \text{ i.e. }\quad \cV_t(S_t) \subseteq \cR_t (S_t, \cV_{t+1}) + \sL_t(C)
	\end{align*}
	as well as
	\begin{align*}
		\cR_t (S_t, \cV_{t+1}) \curlyeqprec^t \cV_t(S_t), \text{ i.e. }\quad \cR_t (S_t, \cV_{t+1}) \subseteq \cV_t(S_t) - \sL_t(C).
	\end{align*}
\end{theorem}
\begin{proof}
	Let us start with proving $\cV_t(S_t) \subseteq \cR_t (S_t, \cV_{t+1}) + \sL_t(C)$.
	Take an element $V\in \cV_t(S_t)$, which corresponds to some feasible strategy $\varphi \in \mathfrak{A}^t(S_t)$. According to the $\preceq$-rectangularity it holds
	\begin{align}\label{eq:bellman10}
		{\vsup\limits_{\theta \in \Theta}}^t  \bE^\theta_t  \left[ {\vsup\limits_{\theta \in \Theta}}^{t+1} \bE^\theta_{t+1} [\ell(S^\varphi_T)]  \right] \preceq^t {\vsup\limits_{\theta \in \Theta}}^t \bE^\theta_t  [\ell(S^\varphi_T)]\preceq^t V. 
	\end{align}
According to Lemma~\ref{lemma:invariantSup} and \ref{lemma:supt+1},  ${\vsup\limits_{\theta \in \Theta}}^t  \bE^\theta_t  \left[ {\vsup\limits_{\theta \in \Theta}}^{t+1} \bE^\theta_{t+1} [\ell(S^\varphi_T)]  \right] \in \cR_t (S_t, \cV_{t+1}) + \sL_t(C)$.
Next, by \eqref{eq:bellman10} combined with convexity of the cone $\sL_t(C)$, we get 
$V\in\cR_t(S_t,\cV_{t+1}) + \sL_t(C)$,
that proves the claim.

Now let us prove $\cR_t (S_t, \cV_{t+1}) \subseteq \cV_t(S_t) - \sL_t(C)$.  For an arbitrary element $R\in \cR_t (S_t, \cV_{t+1})$, there exists some $\varphi_t \in \mathfrak{A}_t (S_t)$ and $X \in \cV_{t+1}(S_{t+1}^{\varphi_t})$, such that $R\in{\Vsup\limits_{\theta\in\Theta}}^t\bE^\theta_t[X]$, and were we recall that  $S_t^{\varphi_t}:=F(t,S_t,\varphi_t, Z_{t+1})$.

Since $X \in \cV_{t+1}(S_{t+1}^{\varphi_t})$,  there exists $\widetilde\varphi \in \mathfrak{A}^{t+1}(S_{t+1}^{\varphi_t})$ such that $X \in {\Vsup\limits_{\theta \in \Theta}}^{t+1}  \bE^\theta_{t+1} [\ell(S^{\widetilde\varphi}_T)]$. Thus $\bar{\varphi} := (\varphi_t, \widetilde \varphi_{t+1}, \dots, \widetilde \varphi_{T-1}) \in \mathfrak{A}^t (S_t)$ and  $S^{\bar{\varphi}}_T = S^{\widetilde\varphi}_T$. Therefore, by $\preceq$-rectangularity, Remark~\ref{rem:vsup-vsup}, and Lemma~\ref{lemma:invariantSup}(a), we have
	\begin{align}\label{eq:bellman20}
		R\preceq^t {\vsup\limits_{\theta \in \Theta}}^t  \bE^\theta_t  \left[ X  \right] = {\vsup\limits_{\theta \in \Theta}}^t  \bE^\theta_t  \left[ {\vsup\limits_{\theta \in \Theta}}^{t+1} \bE^\theta_{t+1} [\ell(S^{\bar{\varphi}}_T)]  \right] \preceq^t {\vsup\limits_{\theta \in \Theta}}^t \bE^\theta_t  [\ell(S^{\bar{\varphi}}_T)]. 
	\end{align}
	Note that by Lemma~\ref{lemma:invariantSup}(b) ${\vsup\limits_{\theta \in \Theta}}^t \bE^\theta_t  [\ell(S^{\bar{\varphi}}_T)] \in \cV_t(S_t) - \sL_t(C)$. Hence, by \eqref{eq:bellman20}, combined with convexity of $\sL_t(C)$, we proved that $R \in V - \sL_t(C)$. The proof is complete. 
\end{proof}

\begin{theorem}\label{th:Bellman3}
	Assume that the family of models $\Theta$ is $\preceq$-rectangular. Then for every time $t = 0, 1, \dots, T-1$ and every state $S_t$  we have
\begin{align*}
	\cB_t (S_t) \preccurlyeq^t \cV_t(S_t), \text{ i.e. }\quad \cV_t(S_t) \subseteq \cB_t (S_t) + \sL_t(C)
\end{align*}
as well as
\begin{align*}
	\cB_t (S_t) \curlyeqprec^t \cV_t(S_t), \text{ i.e. }\quad \cB_t (S_t) \subseteq \cV_t(S_t) - \sL_t(C).
\end{align*}
\end{theorem}
\begin{proof}
We again prove this by a backward recursion. Recall that at time $T$ it holds $\cB_{T} (S_{T}) = \cV_{T} (S_{T})$ by definition. 	

Let us prove $\cV_t(S_t) \subseteq \cB_t(S_t) + \sL_t(C)$. As an induction hypothesis assume $\cV_{t+1}(S_{t+1}) \subseteq \cB_{t+1} (S_{t+1}) + \sL_{t+1}(C)$ holds across all state $S_{t+1}$. Take an element $V\in\cV_t(S_t)$. Then there exists $\varphi=(\varphi_t,\widetilde{\varphi})\in\mathfrak{A}^t(S_t)$, such that $V\in{\Vsup\limits_{\theta \in \Theta}}^t\bE^\theta_t[\ell(S_T^\varphi)]$. 
Hence, $\vsup\limits_{\theta \in \Theta} \bE^\theta_{t+1}  [\ell(S^{\widetilde\varphi}_T)] \in \cV_{t+1} (S^{\varphi_t}_{t+1})+\sL_t(C)$, and by induction hypothesis and convexity of the cone $\sL_t(C)$, we obtain 
$\vsup\limits_{\theta \in \Theta} \bE^\theta_{t+1}  [\ell(S^{\widetilde\varphi}_T)] \in \cB_{t+1} (S^{\varphi_t}_{t+1})+\sL_t(C)$. Consequently, there exists $B \in \cB_{t+1} (S^{\varphi_t}_{t+1})$ such that
	\begin{align*}
		B \preceq^{t+1} \vsup\limits_{\theta \in \Theta} \bE^\theta_{t+1}  [\ell(S^\varphi_T)],
	\end{align*}
	where we used the fact that $S_T^{\varphi} = S_T^{\widetilde{\varphi}}$. By Assumption~\ref{assum:vsup}, Lemmas~\ref{lemma:order_exp}, \ref{lemma:monotone}, and~\ref{lemma:invariantSup}(a) and $\preceq$-rectangularity, 
we have
\[
		\vsup\limits_{\theta \in \Theta} \bE^\theta_t \left[ B \right] \preceq^{t} \vsup\limits_{\theta \in \Theta} \bE^\theta_t \left[  \vsup\limits_{\theta \in \Theta} \bE^\theta_{t+1}  [\ell(S^\varphi_T)] \right] 
		\preceq^{t} \vsup\limits_{\theta \in \Theta} \bE^\theta_t \left[ \ell(S^\varphi_T) \right] \preceq^t V.
\]
From here, by  Lemma~\ref{lemma:invariantSup}(b) $\vsup\limits_{\theta \in \Theta} \bE^\theta_t \left[ B \right] \in \cB_t(S_t) +\sL_t(C)$, and by  transitivity of $\preceq^t$, the claim follows.

Next, let us  prove the second inclusion $\cB_t(S_t) \subseteq \cV_t(S_t) - \sL_t(C)$. As an induction hypothesis assume $\cB_{t+1} (S_{t+1}) \subseteq \cV_{t+1}(S_{t+1}) -\sL_{t+1}(C)$ holds across all state $S_{t+1}$. Take an element $B\in\cB_t(S_t)$, that corresponds to some $\varphi_t \in \mathfrak{A}_t(S_t)$ and $X \in \cB_{t+1}(S^{\varphi_t}_{t+1})$, and by the induction hypothesis, $X\in \cV_{t+1}(S^{\varphi_t}_{t+1}) -\sL_{t+1}(C)$. Then, there exists a strategy $\widetilde\varphi\in\mathfrak{A}^{t+1}(S_{t+1}^{\varphi_t})$ such that 
\begin{align*}
	X \preceq^{t+1} \vsup\limits_{\theta \in \Theta} \bE^\theta_{t+1}  [\ell(S^{{\widetilde\varphi}}_T)].
\end{align*}
Next, we consider the combined strategy $\varphi = (\varphi_t, \widetilde{\varphi})$, and note that $\S_T^{\varphi}=S_T^{\widetilde{\varphi}}$. 
By Lemmas~\ref{lemma:order_exp}, \ref{lemma:monotone}, and~\ref{lemma:invariantSup}(a) and $\preceq$-rectangularity, we deduce 
	\begin{align*}
B\preceq^t		\vsup\limits_{\theta \in \Theta}  \bE^\theta_t [X] \preceq^{t} \vsup\limits_{\theta \in \Theta}  \bE^\theta_t \left[ \vsup\limits_{\theta \in \Theta} \bE^\theta_{t+1}  [\ell(S^{{\varphi}}_T)] \right] \preceq^{t} \vsup\limits_{\theta \in \Theta}  \bE^\theta_t \left[ \ell(S^{{\varphi}}_T) \right].
	\end{align*}
From here, by Lemma~\ref{lemma:invariantSup}(b) $\vsup\limits_{\theta \in \Theta}  \bE^\theta_t \left[ \ell(S^{{\varphi}}_T) \right] \in \cV_t(S_t)- \sL_t(C)$, and by transitivity of $\preceq^t$ the proof is complete. 
\end{proof}

Theorem~\ref{th:Bellman0-1} provided a weaker form of Bellman equation for the robust problem, which did not require further assumptions. Under $\preceq$-rectangularity, Theorems~\ref{th:Bellman0-1} and~\ref{th:Bellman3} jointly provide order relations on the value function $\cV$ and the recursive value function $\cB$ that can be interpreted as a stronger form of Bellman equations for the robust problem.

We conclude this section with a stronger version of Bellman equations, assuming that $\preceq$ is a partial order and the family of models $\Theta$ satisfies the $\preceq$-rectangularity property.

\begin{theorem}\label{th:Bellman4}
	Assume that the preorder $\preceq$ is antisymmetic and the family of models $\Theta$ has the $\preceq$-rectangularity property. Then, 
	\begin{align*}
		\cV_t(S_t) = \cR_t (S_t, \cV_{t+1}) = \cB_t (S_t),
	\end{align*}
	 for every time $t = 0, 1, \dots, T-1$ and every state $S_t$. 
\end{theorem}
\begin{proof}
Recall that in the case of a partial order, the supremum operator $\vsup$ is uniquely defined and the $\preceq$-rectangularity property translates to~\eqref{eq:rect-strong}. Fix time $t$ and a state $S_t$; let us start with proving $\cV_t(S_t) \subseteq \cR_t (S_t, \cV_{t+1})$. Take arbitrary element 
\begin{align*}
		{\vsup\limits_{\theta \in \Theta}}^t  \bE^\theta_t [\ell(S^\varphi_T)] \in \cV_t(S_t)
\end{align*}
	which corresponds to some $\varphi = (\varphi_t, \tilde{\varphi}) \in \mathfrak{A}^t (S_t)$. Note that $\varphi_t \in \mathfrak{A}_t (S_t)$, $\tilde{\varphi} \in \mathfrak{A}^{t+1} (S^{\varphi_t}_{t+1})$ and $S^\varphi_T = S^{\tilde{\varphi}}_T$. According to~\eqref{eq:rect-strong} it holds
	\begin{align*}
		{\vsup\limits_{\theta \in \Theta}}^t  \bE^\theta_t [\ell(S^\varphi_T)] = {\vsup\limits_{\theta \in \Theta}}^t  \bE^\theta_t \left[ {\vsup\limits_{\theta \in \Theta}}^{t+1}  \bE^\theta_{t+1} [\ell(S^{\tilde{\varphi}}_T)] \right].
	\end{align*}
By uniqueness of the supremum operator it follows that ${\vsup\limits_{\theta \in \Theta}}^t  \bE^\theta_t \left[ {\vsup\limits_{\theta \in \Theta}}^{t+1}  \bE^\theta_{t+1} [\ell(S^{\tilde{\varphi}}_T)] \right] \in \cR_t (S_t, \cV_{t+1})$, which proves $\cV_t(S_t) \subseteq \cR_t (S_t, \cV_{t+1})$.\\
	
	Now let us prove $\cR_t (S_t, \cV_{t+1}) \subseteq \cV_t(S_t)$:  Take arbitrary element 
	\begin{align*}
		\vsup\limits_{\theta \in \Theta}  \bE^\theta_t [X] \in \cR_t (S_t, \cV_{t+1}),
	\end{align*}
	 which corresponds to some $\varphi_t \in \mathfrak{A}_t (S_t)$ and $X \in \cV^{t+1} (S^{\varphi_t}_{t+1})$. The element $X \in \cV^{t+1} (S^\varphi_{t+1})$ corresponds to some strategy $\tilde{\varphi} \in \mathfrak{A}^{t+1} (S^{\varphi_t}_{t+1})$ such that
	\begin{align*}
		X = {\vsup\limits_{\theta \in \Theta}}^{t+1}  \bE^\theta_{t+1} [\ell(S^{\tilde{\varphi}}_T)].
	\end{align*}
	Note that the combined strategy $\varphi = (\varphi_t, \tilde{\varphi}) \in \mathfrak{A}^t (S_t)$ generates the terminal state $S^\varphi_T = \S^{\tilde{\varphi}}_T$. Then, according to~\eqref{eq:rect-strong} it holds
	\begin{align*}
		{\vsup\limits_{\theta \in \Theta}}^t  \bE^\theta_t [X] = {\vsup\limits_{\theta \in \Theta}}^t  \bE^\theta_t \left[ {\vsup\limits_{\theta \in \Theta}}^{t+1}  \bE^\theta_{t+1} [\ell(S^{\varphi}_T)] \right] = {\vsup\limits_{\theta \in \Theta}}^t  \bE^\theta_t \left[ \ell(S^{\bar{\varphi}}_T) \right].
	\end{align*}
	By uniqueness of the supremum operator it follows that ${\vsup\limits_{\theta \in \Theta}}^t  \bE^\theta_t \left[ \ell(S^{{\varphi}}_T) \right] \in \cV_t(S_t)$, which proves $\cR_t (S_t, \cV_{t+1}) \subseteq \cV_t(S_t)$.
	
	The remaining claim follows inductively from the recursive definition of $\cB_t (S_t)$.

\end{proof}

\section{Dynamic programming under the component-wise partial order}
\label{sec:component-wise}

This section is dedicated to a particular, but important, case of the partial order $\leq$ corresponding to the ordering cone $C=\bR_+^d$. This partial order, as well as its extension $\leq^t$ to the space of random vectors, corresponds to the natural component-wise comparison of (random) vectors, $X\leq Y$ iff for all $i=1,\ldots,d$ it holds    $X_i\leq Y_i$, a.s.  

We recall from Definition~\ref{def:supremum0} that the vector-valued supremum is (i) an \textit{upper bound} and (ii) the \textit{smallest} among  upper bounds. In case of the component-wise order we not only have existence and uniqueness, but also an explicit formula.
\begin{lemma}
For arbitrary time $t \in \cT'$ and arbitrary collection of random vectors $\{X^\theta\}_{\theta \in \Theta} \subseteq \sL_t (\bR^d)$ the supremum with respect to the partial order $\leq^t$ exists, is unique and given by
\begin{align}
\label{eq:vsup-component}
		\left( \vsup\limits_{\theta\in \Theta} X^\theta \right)_i (\omega) := \max\limits_{\theta \in \Theta} X_i^\theta (\omega), 
\end{align}
for $\omega \in \Omega$ and $i \in \{1, \dots, d\}$.
\end{lemma}
\begin{proof}
Existence follows from Lemma~\ref{lemma:cone2Sufficeint}, uniqueness follows from $\leq^t$ being a partial order and Lemma~\ref{lemma:invariantSup}. Validity of~\eqref{eq:vsup-component} can be verified through properties in Definition~\ref{def:supremum0}.
\end{proof}

\begin{remark}
The approach of this paper is motivated by the stream of literature on (static) robust multi-objective optimization using a vector-valued (or ideal point) notion of supremum of a collection of vectors. In particular, the supremum operator \eqref{eq:vsup-component} under the component-wise partial order corresponds to the \textit{robustified objective} proposed in \cite{FliegeWerner2014}. Results of this section illustrate that the notion naturally extends to the dynamic setting and measurability is preserved.
\end{remark}

It is worth mentioning that similar statements and constructions of supremum hold true for any (solid, convex) ordering cone $C$ in dimension $d=2$, see Remark~\ref{rem:2d-sup}. Hence, the important case of problems with dual objectives can be addressed using a similar approach as outlined in this section.

Assumptions on the dynamics of the controlled process $(S_t)_{t \in\cT}$ remain unchanged and 
our interest remains to be dynamic programming for the robust multi-objective stochastic control problem, 
\begin{align}
	\label{prob_robust2}
	\begin{split}
	\text{minimize }& \quad {\vsup\limits_{\theta \in \Theta}}^t\, \bE^\theta_t [\ell(S^\varphi_T)] \quad \text{ with respect to } \leq^t \\
	\text{subject to }& \quad \varphi \in \mathfrak{A}^t(S_t). 
	\end{split}
\end{align}
From Section~\ref{sec:dyn-prog-preorder} we know that the family of robust multi-objective problems satisfies the weak set-valued Bellman's principle and under a rectangularity assumption on the family $\Theta$ it also satisfies strong set-valued Bellman's principle. Verifying $\preceq$-rectangularity of $\Theta$ under a general preorder $\preceq$ as well as constructing a $\preceq$-rectangular family of market models is an open challenge. However, the situation simplifies for the component-wise order, where known results on (scalar) rectangularity can be directly used.

\begin{definition}[\cite{Shapiro2016}]
\label{def:rectangle-1d}
We say that the family of models $\Theta$ is \textit{$m$-rectangular} if for all times $t=0, 1, \dots, T-1$ and all random variables $X \in \sL_T(\bR)$ it holds
	\begin{align}\label{eq:rect-1d}
		{\sup\limits_{\theta \in \Theta}}^t  \bE^\theta_t  \left[ {\sup\limits_{\theta \in \Theta}}^{t+1} \bE^\theta_{t+1} [X]  \right] 
		=  {\sup\limits_{\theta \in \Theta}}^t \bE^\theta_t  [X].
	\end{align}
\end{definition} 
\begin{lemma}
The family of models $\Theta$ is $m$-rectangular if and only if it is $\leq$-rectangular.
\end{lemma}
\begin{proof}
Both implications follow from the component-wise construction~\eqref{eq:vsup-component} of supremum with respect to $\leq^t$.
\end{proof}

\begin{remark}
\cite{Shapiro2016} provides an approach for recursive constructions of $m$-rectangular family of probability measures from marginals. The same approach can, therefore, be used to construct $\leq$-rectangular family $\Theta$.
\end{remark}

Let us now summarize the results derived in Section~\ref{sec:dyn-prog-preorder} within a context of coordinate-wise order. Recall that  central role is played by the value function
\begin{align*}
	\cV_t (S_t) &= \bigcup\limits_{\varphi\in\mathfrak{A}} \vsup_{\theta\in \Theta} \bE^{\theta}_t \left[ \ell(S^\varphi_T) \right]
\end{align*}
consisting of the suprema over all feasible strategies, alongside the recursive value function $\cB_t(S_t)$.

\begin{corollary}\label{prop:Bellman-comp}
For every time $t \in\cT'$ and every state $S_t$  we have
	\begin{align*}
		\cB_t (S_t) \subseteq \cV_t(S_t) + \sL_t(\bR^d_+) \quad \quad \text{and } \quad \quad \cV_t(S_t) \subseteq \cB_t (S_t) - \sL_t(\bR^d_+).
	\end{align*}
If, additionally, the set of models $\Theta$ is $m$-rectangular, then for every time $t\in\cT'$ and every state $S_t$
	\begin{align*}
		\cV_t(S_t) = \cB_t (S_t) = \left\lbrace \vsup_{\theta\in \Theta} \bE^{\theta}_t \left[ X \right] \; \vert \; \varphi_t \in \mathfrak{A}_t(S_t), X \in \cV_{t+1}( F(t,S_t,\varphi_t, Z_{t+1}))  \right\rbrace.
	\end{align*}
\end{corollary}

As outlined in Remark~\ref{rem:upper-im}, an alternative candidate for the value function of a multi-objective control problem is (a set-valued mapping attaining) the upper image of the problem. Here we shortly explore dynamic programming results within the context of an upper image. 
In what follows we assume that the topology is generated by the Euclidean norm. The upper image of the robust multi-objective problem~\eqref{prob_robust2} is
\begin{align*}
\begin{split}
	\cP_t (S_t) &=  \cl \bigcup\limits_{\varphi\in\mathfrak{A}^t(S_t)} \vsup_{\theta\in \Theta} \bE^{\theta}_t \left[ \ell(S^\varphi_T) \right] + \sL_t(\bR^d_+) = \cl \left( \cV_t(S_t) + \sL_t(\bR^d_+) \right) \\
	&= \cl \left\lbrace X \in \sL_t(\bR^d) \quad \vert \quad \exists \varphi \in \mathfrak{A}^t(S_t) \quad \forall \theta\in \Theta : \quad \bE^{\theta}_t \left[ \ell(S^\varphi_T) \right] \leq X \right\rbrace.
\end{split}
\end{align*}
Analogous to the results for the value function $\cV$, the next result gives a weak, as well as strong, set-valued Bellman's principle for the robust problem in terms of the upper image.
\begin{proposition}\label{prop:Bellman-comp2}
For every time $t \in\cT'$ and every state $S_t$  we have
\begin{align}
\label{eq:P1}
\left\lbrace \vsup_{\theta\in \Theta} \bE^{\theta}_t \left[ X \right] \; \vert \; \varphi_t \in \mathfrak{A}_t(S_t), X \in \cP_{t+1}( F(t,S_t,\varphi_t, Z_{t+1}))  \right\rbrace \subseteq \cP_t(S_t).
\end{align}
If, additionally, the set of models $\Theta$ is $m$-rectangular, then for every time $t\in\cT'$ and every state $S_t$
\begin{align}
\label{eq:P2}
\cP_t(S_t) =   \cl \left\lbrace \vsup_{\theta\in \Theta} \bE^{\theta}_t \left[ X \right] \; \vert \; \varphi_t \in \mathfrak{A}_t(S_t), X \in \cP_{t+1}( F(t,S_t,\varphi_t, Z_{t+1}))  \right\rbrace.
\end{align}
\end{proposition}
\begin{proof}
According to \eqref{eq:valueF1} it holds
\begin{align*}
\left\lbrace \vsup_{\theta\in \Theta} \bE^{\theta}_t \left[ X \right] \; \vert \; \varphi_t \in \mathfrak{A}_t(S_t), X \in \cV_{t+1}( S_{t+1}^{\varphi_t})  \right\rbrace \subseteq \cV_t(S_t) + \sL_t(\bR^d_+) \subseteq \cP_t(S_t).
\end{align*}
Then, an application of Lemma~\ref{lemma:monotone} implies 
\begin{align}
\label{eq:proofP1}
\begin{split}
&\left\lbrace \vsup_{\theta\in \Theta} \bE^{\theta}_t \left[ X \right] \; \vert \; \varphi_t \in \mathfrak{A}_t(S_t), X \in \cV_{t+1}( S_{t+1}^{\varphi_t}) + \sL_{t+1}(\bR^d_+)  \right\rbrace  \\ &\subseteq
\left\lbrace \vsup_{\theta\in \Theta} \bE^{\theta}_t \left[ X \right] \; \vert \; \varphi_t \in \mathfrak{A}_t(S_t), X \in \cV_{t+1}( S_{t+1}^{\varphi_t})  \right\rbrace + \sL_t(\bR^d_+) \subseteq \cP_t(S_t).
\end{split}
\end{align}
We denote by $\mathbf{1}$ a $d$-dimensional vector constantly equal to one. Fix some $\varphi_t \in \mathfrak{A}_t(S_t)$ and take arbitrary $X \in \cP_{t+1} (S_{t+1}^{\varphi_t})$. Then for all $n \in \bN$ it holds $X^n := X + \frac{1}{n} \mathbf{1} \in \cP_{t+1} (S_{t+1}^{\varphi_t}) + \text{int } \sL_{t+1}(\bR^d_+)  \subseteq  \cV_{t+1}( S_{t+1}^{\varphi_t}) + \sL_{t+1}(\bR^d_+)$ and, therefore, by~\eqref{eq:proofP1} it follows $\vsup_{\theta\in \Theta} \bE^{\theta}_t \left[ X^n \right] \in \cP_t(S_t)$. At the same time, by Lemma~\ref{lemma:invariantSup}(d) it holds
\begin{align*}
\vsup_{\theta\in \Theta} \bE^{\theta}_t \left[ X^n \right] = \vsup_{\theta\in \Theta} \bE^{\theta}_t \left[ X + \frac{1}{n} \mathbf{1} \right] = \vsup_{\theta\in \Theta} \left( \bE^{\theta}_t \left[ X  \right] + \frac{1}{n} \mathbf{1} \right) = \vsup_{\theta\in \Theta} \left( \bE^{\theta}_t \left[ X  \right]  \right) + \frac{1}{n}\mathbf{1}.
\end{align*}
Therefore, the sequence  $\left\lbrace \vsup_{\theta\in \Theta} \bE^{\theta}_t \left[ X^n \right] \right\rbrace_{n \in \bN} \subseteq \cP_t(S_t)$ converges (in the norm-topology) to $\vsup_{\theta\in \Theta}\bE^{\theta}_t \left[ X  \right]$. Since  the upper image $\cP_t (S_t)$ is closed, it holds $\vsup_{\theta\in \Theta}\bE^{\theta}_t \left[ X  \right] \in \cP_t(S_t)$, which proves \eqref{eq:P1}.

Now let $\Theta$ be ($m$- and equivalently $\leq$-)rectangular. The inclusion $\supseteq$ in \eqref{eq:P2} follows from \eqref{eq:P1} and $\cP_t(S_t)$ being a closed set. From Corollary~\ref{prop:Bellman-comp}, inclusion $\sL_t(\bR^d_+) \subseteq \sL_{t+1} (\bR^d_+)$ and Lemma~\ref{lemma:monotone} we obtain
\begin{align*}
\cV_t(S_t) + \sL_t(\bR^d_+) &= \left\lbrace \vsup_{\theta\in \Theta} \bE^{\theta}_t \left[ X \right] \; \vert \; \varphi_t \in \mathfrak{A}_t(S_t), X \in \cV_{t+1}( S_{t+1}^{\varphi_t}) + \sL_{t+1}(\bR^d_+)  \right\rbrace \\
& \subseteq \cl \left\lbrace \vsup_{\theta\in \Theta} \bE^{\theta}_t \left[ X \right] \; \vert \; \varphi_t \in \mathfrak{A}_t(S_t), X \in \cP_{t+1}( F(t,S_t,\varphi_t, Z_{t+1}))  \right\rbrace.
\end{align*}
Then for arbitrary $Y \in \cP_t(S_t)$ we can define a sequence $\left\lbrace Y + \frac{1}{n} \mathbf{1} \right\rbrace_{n \in \bN} \subseteq \cV_t(S_t) + \sL_t(\bR^d_+)$ which is contained in the right-hand side of \eqref{eq:P2} and converges to $Y$. Therefore, the result follows by the right-hand side of \eqref{eq:P2} being closed.
\end{proof}

Note that the above proof did not make direct use of the specific form of the ideal point supremum~\eqref{eq:vsup-component} under the partial order $\leq^t$. Therefore, a version of Proposition~\ref{prop:Bellman-comp2} (based on $\preceq$-rectangularity and set orders induced by $\preceq$) also holds within the setup of Section~\ref{sec:dyn-prog-preorder}.

\section{Example}
In this section we present some illustrative examples that depict some key points of the study. 

\begin{example}\label{ex:binomial1}

In this example, we illustrate the importance of rectangularity property. 
Consider a two period model $t\in\set{0,1,2}$, and a binomial tree  setup. That is, $\Omega=\set{\omega^1,\omega^2, \omega^3, \omega^4}=\set{ (\omega_1, \omega_2) : \omega_1, \omega_2 \in \{u, d\} }$ that is endowed with filtration $\set{\sF_t}$
 generated by the stochastic factor $Z$  that goes up, or down at each time step, perhaps with different transition probabilities at different steps and/or nodes, see the diagram in Figure~\ref{fig:binomial1}.
A probability measure $\bP$ on this filtered probability space is identified by the  quantities 
\begin{align*}
p_u = \bP(\omega_1=u), \quad p_{u|u} = \bP(\omega_2=u|\omega_1=u), \quad p_{u|d} = \bP(\omega_2=u|\omega_1=d),
\end{align*} 
which in turn would determine uniquely the (marginal) distributions of the stochastic factor $Z$.

\tikzstyle{level 1}=[level distance=2.2cm, sibling distance=2cm,->]
\tikzstyle{level 2}=[level distance=2.2cm, sibling distance=1cm,->]

\tikzstyle{bag} = [text width=2em, text centered]
\tikzstyle{end} = []

\begin{figure}
	\centering 
	\begin{tikzpicture}[grow=right, sloped]
		\node[bag] {  }
		child {
			node[bag] {  }        
			child {
				node[end, label=right:
				{ $\omega^4 = (d,d)$ }] {}
				edge from parent
				node[above] {}
				node[below]  {  }
			}
			child {
				node[end, label=right:
				{ $\omega^3=(d,u)$ }] {}
				edge from parent
				node[above] {}
				node[above]  {$p_{u|d}$}
			}
			edge from parent 
			node[above] {}
			node[below]  {  }
		}
		child {
			node[bag] {   }        
			child {
				node[end, label=right:
				{  $\omega^2=(u,d)$  }] {}
				edge from parent
				node[above] {  }
				node[below]  {}
			}
			child {
				node[end, label=right:
				{ $\omega^1=(\omega_1 = u, \omega_2 = u)$ }] {}
				edge from parent
				node[above] {$p_{u|u}$}
				node[below]  {}
			}
			edge from parent         
			node[above] {$p_u$}
			node[below]  {}
		};
	\end{tikzpicture}
	\caption{\label{fig:binomial1} Tree diagram for filtered probability space in Example~\ref{ex:binomial1}. 
		}
\end{figure}
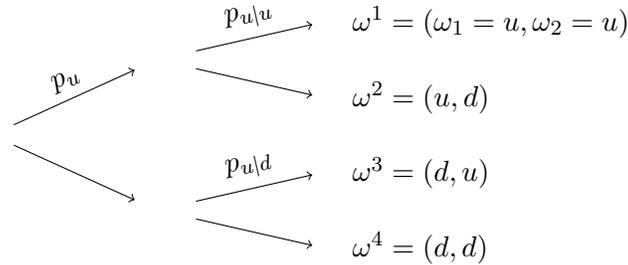

Let us consider two models, given by
\begin{align*}
\bP_{\theta_1}: & \quad 	p_u^1 = 1/4, \quad p_{u|u}^1 = p_{u|d}^1=1/2 \\
\bP_{\theta_2}:	& \quad p_u^2 = 1/2, \quad p_{u|u}^2 = p_{u|d}^2=3/4. 
\end{align*} 
We use $\bP_{\theta_1}$ and $\bP_{\theta_2}$ to construct an $m$-rectangular family of models $\Theta = \{ \theta_1, \dots,  \theta_8 \}$,  given in Table~\ref{tab:1}, for example by following \cite{Shapiro2016} and build probabilities by exhausting all possible combinations of conditional probabilities.  Additionally,   we also take (smaller) family of models $\Theta^0 = \set{\theta_1, \theta_2, \theta_5, \theta_8}$. Family $\Theta^0$ collects models under which the the events $\set{\omega_1=u}$ and $\set{\omega_2=u}$ are independent. We note that $\Theta^0$ is not $m$-rectangular.

\begin{table}[htbp]
\centering 
	\begin{tabular}{|c |  c c c c   c c c c|}
	\hline 								
	& $\theta_1$ 	& $\theta_2$ 	& $\theta_3$ 	& $\theta_4$ 	& $\theta_5$ 	& $\theta_6$ 	& $\theta_7$ 	& $\theta_8$ \\ \hline
	$p^\theta_u$		
	& $\frac{1}{4}$	&  $\frac{1}{2}$	& $\frac{1}{2}$	& $\frac{1}{2}$	& $\frac{1}{2}$	& $\frac{1}{4}$	& $\frac{1}{4}$	& $\frac{1}{4}$	\\[5pt]
	$p^\theta_{u|u}$ 
	& $\frac{1}{2}$	&  $ \frac{3}{4}$	& $\frac{3}{4}$	& $\frac{1}{2}$	& $\frac{1}{2}$	& $\frac{1}{2}$	& $\frac{3}{4}$	& $\frac{3}{4}$	\\[5pt]
	$p^\theta_{u|d}$ 
	& $\frac{1}{2}$	&  $\frac{3}{4}$& $\frac{1}{2}$	& $\frac{3}{4}$	& $\frac{1}{2}$	& $\frac{3}{4}$	& $\frac{1}{2}$	& $\frac{3}{4}$	\\ \hline
\end{tabular}
\caption{Set of probability measures satisfying $m$-rectangularity property.}
\label{tab:1}
\end{table}

Assume that there are only two admissible strategies $\varphi, \psi\in\mathfrak{A}$ that generate two possible outcomes $S_T^\varphi, S_T^\psi$ of the two-dimensional terminal state $S_T\in\sL_T(\bR^2)$, with specific values given in Table~\ref{tab:2}. For the sake of brevity we omit here presenting the controlled dynamics or the dynamics of the stochastic factors. 
Within this example we use the identify  function $\ell: \bR^2 \to \bR^2$, $\ell(S) = S$ as the multi-loss function, and we assume the component-wise order $\leq$.

\begin{table}
	\centering
	\begin{tabular}{|c |  c c c c  |}
\hline
		& $\omega^1$ & $\omega^2$ & $\omega^3$ & $\omega^4$ \\ 
	\hline
		$S_T^\varphi(\omega^i)$ &  $\begin{pmatrix} 8\\ 0\end{pmatrix}$ & $\begin{pmatrix}0 \\8 \end{pmatrix}$ &  $\begin{pmatrix} 0\\0 \end{pmatrix}$  & $\begin{pmatrix} 8\\ 8\end{pmatrix}$   \\ 
		$S_T^\psi(\omega^i)$ & $\begin{pmatrix} 0\\ 8 \end{pmatrix}$ & $\begin{pmatrix} 0\\0 \end{pmatrix}$ &$\begin{pmatrix} 6\\ 0\end{pmatrix}$ & $\begin{pmatrix} 6\\ 8\end{pmatrix}$ \\
		\hline
	\end{tabular}
	\caption{Terminal values corresponding to two strategies.}
	\label{tab:2}
\end{table}

By direct computation we find $\bE^\theta_t[\ell(S_T^\varphi)]$ as well as $\bE^\theta_t[\ell(S_T^\psi)]$ for all $\theta\in\Theta$ and $t = 0,1$; see Tables~\ref{tab:3} and~\ref{tab:4}. Recalling the component-wise structure of the supremum operator~\eqref{eq:vsup-component} and applying further direct computation, we can verify that for $\bar{\Theta}\in\set{\Theta, \Theta^0}$ it holds
\begin{align*}
	\vsup\limits_{\theta \in \bar{\Theta}}  \bE_1^\theta [\ell(S^{\varphi}_T)] = 
	\begin{cases} 
		\begin{pmatrix}  6\\ 4\end{pmatrix}  & \omega_1 = u \\ 
			\begin{pmatrix} 4\\ 4\end{pmatrix} & \omega_1 = d \end{cases}  
			\quad \quad \text{ and } \quad \quad
\vsup\limits_{\theta \in \bar{\Theta}}  \bE_1^\theta [\ell(S^{\psi}_T)] 
	= \begin{cases} 
		\begin{pmatrix} 0\\ 6\end{pmatrix}  & \omega_1 = u \\ \begin{pmatrix} 6\\ 4\end{pmatrix} & \omega_1 = d 
		\end{cases}
\end{align*}
as well as 
\begin{align*}
	\vsup\limits_{\theta \in \bar{\Theta}} \bE_0^{\theta}  \left[ \vsup\limits_{\theta \in \bar{\Theta}}  \bE_1^\theta [\ell(S^{\varphi}_T)] \right] 
		= \begin{pmatrix} 5 \\ 4 \end{pmatrix}
		\quad \quad \text{ and } \quad \quad
	\vsup\limits_{\theta \in \bar{\Theta}} \bE_0^{\theta}  \left[ \vsup\limits_{\theta \in \bar{\Theta}}  \bE_1^\theta [\ell(S^{\psi}_T)] \right] 
	= \begin{pmatrix} 4.5 \\ 5 \end{pmatrix}.
\end{align*}
This shows that for both families of models $\bar{\Theta}\in\set{\Theta, \Theta^0}$, the recursive value function is
\begin{align*}
	{\cB}_0^{\bar{\Theta}}(S_0) = \left\lbrace \vsup\limits_{\theta \in \bar{\Theta}}  \bE_0^\theta \left[ \vsup\limits_{\theta \in \bar{\Theta}}  \bE_1^\theta [\ell(S^{\varphi}_T)] \right],  \vsup\limits_{\theta \in \bar{\Theta}}  \bE_0^\theta \left[ \vsup\limits_{\theta \in \bar{\Theta}}  \bE_1^\theta [\ell(S^{\psi}_T)] \right]  \right\rbrace  
	= \left\lbrace \begin{pmatrix} 5 \\ 4  \end{pmatrix},  \begin{pmatrix} 4.5 \\ 5  \end{pmatrix}  \right\rbrace .
\end{align*}
However, we do not obtain the same value functions $\cV$. For the full set of models $\Theta$ we get 
\begin{align*}
	\cV_0^{\Theta} (S_0) = \left\lbrace \vsup\limits_{\theta \in \Theta}  \bE_0^\theta [\ell(S^{\varphi}_T)],  \ \vsup\limits_{\theta \in \Theta}  \bE_0^\theta [\ell(S^{\psi}_T)]  \right\rbrace  
	= \left\lbrace \begin{pmatrix} 5 \\ 4  \end{pmatrix},  \begin{pmatrix} 4.5 \\ 5  \end{pmatrix}  \right\rbrace, 
\end{align*}
while for family $\Theta^0$ we obtain
\begin{align*}
\cV_0^{\Theta^0} (S_0) = \left\lbrace \vsup\limits_{\theta \in \Theta^0}  \bE_0^\theta [\ell(S^{\varphi}_T)],  \ \vsup\limits_{\theta \in \Theta^0}  \bE_0^\theta [\ell(S^{\psi}_T)]  \right\rbrace   
 = \left\lbrace \begin{pmatrix} 4 \\ 4  \end{pmatrix},  \begin{pmatrix} 4.5 \\ 4  \end{pmatrix}  \right\rbrace.
\end{align*}

This illustrates the difference between (non-rectangular) family of models $\Theta^0$ and $\leq$-rectangular family of models $\Theta$. Only the weak set-valued Bellman's principle  is satisfied for $\Theta^0$, it holds $\cB_0^{\Theta^0} \subsetneq \cV_0^{\Theta^0} + \bR^2_+$ and $\cV_0^{\Theta^0} \subsetneq \cB_0^{\Theta^0} - \bR^2_+$. For the $\leq$-rectangular family $\Theta$, the strong set-valued Bellman's principle holds as $\cV_0^{\Theta} = \cB_0^{\Theta}$.

\begin{table}
\centering 
	\begin{tabular}{|c |  c c c c   c c c c  c|}
		\hline
		$\theta$ 								
		& $\theta_1$ 	& $\theta_2$ 	& $\theta_3$ 	& $\theta_4$ 	& $\theta_5$ 	& $\theta_6$ 	& $\theta_7$ 	& $\theta_8$ & \\ \hline
		\multirow{2}{*}{$\bE_1^\theta [\ell(S^{\varphi}_2)]$}
		& $\begin{pmatrix} 4\\ 4\end{pmatrix} $  & $\begin{pmatrix} 6\\ 2\end{pmatrix}$  & $\begin{pmatrix} 6\\ 2\end{pmatrix}$  & $\begin{pmatrix} 4\\ 4\end{pmatrix}$  & \multirow{2}{*}{$\theta_1$}  & \multirow{2}{*}{$\theta_4$}  & \multirow{2}{*}{$\theta_3$}  & \multirow{2}{*}{$\theta_2$} & $\omega_1 = u$ \\
		& $\begin{pmatrix} 4\\ 4\end{pmatrix}$  & $\begin{pmatrix} 2\\ 2\end{pmatrix}$  & $\begin{pmatrix} 4\\ 4\end{pmatrix}$   & $\begin{pmatrix} 2\\ 2\end{pmatrix}$  & & & & & $\omega_1 = d$  \\ \hline
		$\bE_0^\theta [\ell(S^{\varphi}_2)]$
		&  $\begin{pmatrix} 4\\ 4\end{pmatrix}$  &  $\begin{pmatrix} 4\\ 2\end{pmatrix}$  &  $\begin{pmatrix} 5\\ 3\end{pmatrix}$  &  $\begin{pmatrix} 3\\ 3\end{pmatrix}$  &  $\begin{pmatrix} 4\\ 4\end{pmatrix}$  &  $\begin{pmatrix} 2.5\\ 2.5\end{pmatrix}$  &  $\begin{pmatrix} 4.5\\ 3.5\end{pmatrix}$  &  $\begin{pmatrix} 3\\ 2\end{pmatrix}$ &  \\ \hline 
	\end{tabular}
	\caption{Conditional expectations for different models corresponding to strategy $\varphi$.}
	\label{tab:3}
\end{table}

\begin{table}
	\centering
	\begin{tabular}{|c |  c c c c   c c c c  c|}
		\hline
		$\theta$ 								
		& $\theta_1$ 	& $\theta_2$ 	& $\theta_3$ 	& $\theta_4$ 	& $\theta_5$ 	& $\theta_6$ 	& $\theta_7$ 	& $\theta_8$ & \\ \hline
		\multirow{2}{*}{$\bE_1^\theta [\ell(S^{\psi}_T)]$}
		& $\begin{pmatrix} 0\\ 4\end{pmatrix} $  & $\begin{pmatrix} 0\\ 6\end{pmatrix}$  & $\begin{pmatrix} 0\\ 6\end{pmatrix}$  & $\begin{pmatrix} 0\\ 4\end{pmatrix}$  & \multirow{2}{*}{$\theta_1$}  & \multirow{2}{*}{$\theta_4$}  & \multirow{2}{*}{$\theta_3$}  & \multirow{2}{*}{$\theta_2$} & $\omega_1 = u$ \\
		& $\begin{pmatrix} 6\\ 4\end{pmatrix}$  & $\begin{pmatrix} 6\\ 2\end{pmatrix}$  & $\begin{pmatrix} 6\\ 4\end{pmatrix}$   & $\begin{pmatrix} 6\\ 2\end{pmatrix}$  & & & & & $\omega_1 = d$  \\ \hline
		$\bE_0^\theta [\ell(S^{\psi}_T)]$
		&  $\begin{pmatrix} 4.5\\ 4\end{pmatrix}$  &  $\begin{pmatrix} 3\\ 4\end{pmatrix}$  &  $\begin{pmatrix} 3\\ 5\end{pmatrix}$  &  $\begin{pmatrix} 3\\ 3\end{pmatrix}$  &  $\begin{pmatrix} 3\\ 4\end{pmatrix}$  &  $\begin{pmatrix} 4.5\\ 2.5\end{pmatrix}$  &  $\begin{pmatrix} 4.5\\ 4.5\end{pmatrix}$  &  $\begin{pmatrix} 4.5\\ 3\end{pmatrix}$ &  \\ \hline 
	\end{tabular}
\caption{Conditional expectations  for different models corresponding to strategy $\psi$. }
\label{tab:4}
\end{table}

\end{example}

\begin{example} We present an application  to portfolio optimization problem.  Consider an investor that holds a collection of portfolios, with wealth process represented by a vector $S_t\varphi=(S_t^{1,\varphi},\ldots,S_t^{d,\varphi})\in\bR^d$, where $S_t^{j,\varphi}$ is the wealth of the $j$-th portfolio corresponding to a self-financing trading strategy $\varphi$ that additionally may satisfy various constraints, such as short-selling constraints, turn-over constraints, etc. For example, each component could be a portfolio in a different currency, or in a different market with fundamentally different risk characteristics. See, for instance, \cite[Section~5]{FeinsteinRudloff2013a} for more details and motivations for studying a similar setup, instead of the traditional scalar formulations of the problem.   
The investor may have different preferences for each market, and hence use different utility functions $\ell_j$ for each component $S^j$, and overall optimizing the  multi-objective 
$ \left[\bE^\theta[\ell_1(S_T^\varphi)], \ldots, \bE^\theta[\ell_d(S_T^\varphi)]  	\right]^\mT$. 
Our motivation to study this formulation of portfolio optimization stems from the emerging fintech theme of digital financial advice for investment management and trading, known as robo-advising; see the survey \cite{DAcuntoRossi2021}. Traditionally, the robo-advising problem is formulated as one period mean-variance Markowitz portfolio optimization or  its variations such as Black-Litterman model \cite{KoEtAl2023}. The investor's risk profile (risk-aversion or risk tolerance coefficient) is elicited through questionnaires and assumed to be fixed and known between interaction times with the robo-advisor. Arguably, for unsophisticated and small investors, who constitute the vast majority of robo-advising platform users, finding the risk tolerance coefficient is a notoriously difficult problem, especially when combining several markets or classes of assets \cite{AlsabahCapponi2020}. A multi-valued formulation would allow: (a) eliciting the risk tolerance for major portfolio components (equity market, emerging markets, real estate market) and viewing the optimal portfolio as a point on the efficient frontier, potentially chosen by the investor through another questionnaire, and (b) dealing with inherently time-inconsistent problems. Model uncertainty would permit the use of tractable models while assuming uncertainty about the parametric characteristics of the driving stochastic factors. Detailed numerical  implementation of these ideas using market data is outside the scope of this manuscript.

\end{example}

\section*{Acknowledgment}
The authors acknowledge support from the Society for Industrial and Applied Mathematics (SIAM) as part of the SIAM Postdoctoral Support Program, which is funded by contributions to the SIAM Postdoctoral Support Fund, established by a gift from Drs. Martin Golubitsky and Barbara Keyfitz. 
Part of this research was performed while the authors were visiting the Institute for Mathematical and
Statistical Innovation (IMSI), which is supported by the US National Science Foundation. IC acknowledge partial support from the US National Science Foundation  Grant DMS-2407549.

\bibliographystyle{alpha}

\begin{thebibliography}{ACRLS20}
	
	\bibitem[ACRLS20]{AlsabahCapponi2020}
	H.~Alsabah, A.~Capponi, O.~Ruiz~Lacedelli, and M.~Stern.
	\newblock {Robo-Advising: Learning Investors Risk Preferences via Portfolio
		Choices}.
	\newblock {\em Journal of Financial Econometrics}, 19(2):369--392, January
	2020.
	
	\bibitem[AF20]{AraratFeinstein2020}
	\c{C}a\u{g}{\i}n Ararat and Zachary Feinstein.
	\newblock Set-valued risk measures as backward stochastic difference inclusions
	and equations.
	\newblock {\em Finance and Stochastics}, 25(1):43--76, December 2020.
	
	\bibitem[BCC{\etalchar{+}}19]{BCCCJ2019}
	Tomasz~R. Bielecki, Tao Chen, Igor Cialenco, Areski Cousin, and Monique
	Jeanblanc.
	\newblock Adaptive robust control under model uncertainty.
	\newblock {\em {SIAM} Journal on Control and Optimization}, 57(2):925--946,
	January 2019.
	
	\bibitem[DR21]{DAcuntoRossi2021}
	Francesco D'Acunto and Alberto~G. Rossi.
	\newblock {\em Robo-Advising}, pages 725--749.
	\newblock Springer International Publishing, Cham, 2021.
	
	\bibitem[EIS14]{EhrgottEtAl2014}
	Matthias Ehrgott, Jonas Ide, and Anita Sch\"obel.
	\newblock Minmax robustness for multi-objective optimization problems.
	\newblock {\em European Journal of Operational Research}, 239(1):17--31,
	November 2014.
	
	\bibitem[FR13]{FeinsteinRudloff2013a}
	Zachary Feinstein and Birgit Rudloff.
	\newblock Time consistency of dynamic risk measures in markets with transaction
	costs.
	\newblock {\em Quantitative Finance}, 13(9):1473--1489, September 2013.
	
	\bibitem[FR14]{FeinsteinRudloff2015}
	Zachary Feinstein and Birgit Rudloff.
	\newblock Multi-portfolio time consistency for set-valued convex and coherent
	risk measures.
	\newblock {\em Finance and Stochastics}, 19(1):67--107, October 2014.
	
	\bibitem[FR17]{FeinsteinRudloff2017}
	Zachary Feinstein and Birgit Rudloff.
	\newblock A recursive algorithm for multivariate risk measures and a set-valued
	{B}ellman's principle.
	\newblock {\em Journal of Global Optimization}, 68(1):47--69, August 2017.
	
	\bibitem[FR21]{FeinsteinRudloff2019}
	Zachary Feinstein and Birgit Rudloff.
	\newblock Time consistency for scalar multivariate risk measures.
	\newblock {\em Statistics: Risk Modeling}, 38(3-4):71--90, July 2021.
	
	\bibitem[FRZ22]{FeinsteinEtAl2022}
	Zachary Feinstein, Birgit Rudloff, and Jianfeng Zhang.
	\newblock Dynamic set values for nonzero-sum games with multiple equilibriums.
	\newblock {\em Mathematics of Operations Research}, 47(1):616--642, February
	2022.
	
	\bibitem[FW14]{FliegeWerner2014}
	J\"org Fliege and Ralf Werner.
	\newblock Robust multiobjective optimization \& applications in portfolio
	optimization.
	\newblock {\em European Journal of Operational Research}, 234(2):422--433,
	April 2014.
	
	\bibitem[HHL{\etalchar{+}}15]{HamelEtAl2015}
	Andreas~H. Hamel, Frank Heyde, Andreas L\"ohne, Birgit Rudloff, and Carola
	Schrage.
	\newblock {\em Set Optimization--A Rather Short Introduction}, pages 65--141.
	\newblock Springer Berlin Heidelberg, 2015.
	
	\bibitem[HV20]{HamelVisetti2020}
	Andreas~H Hamel and Daniela Visetti.
	\newblock The value functions approach and {H}opf-{L}ax formula for
	multiobjective costs via set optimization.
	\newblock {\em Journal of Mathematical Analysis and Applications},
	483(1):123605, March 2020.
	
	\bibitem[IKK{\etalchar{+}}14]{IdeEtAl2014}
	Jonas Ide, Elisabeth K\"obis, Daishi Kuroiwa, Anita Sch\"obel, and Christiane
	Tammer.
	\newblock The relationship between multi-objective robustness concepts and
	set-valued optimization.
	\newblock {\em Fixed Point Theory and Applications}, 2014(1), March 2014.
	
	\bibitem[IS15]{IdeSchoebel2015}
	Jonas Ide and Anita Sch\"obel.
	\newblock Robustness for uncertain multi-objective optimization: a survey and
	analysis of different concepts.
	\newblock {\em OR Spectrum}, 38(1):235--271, October 2015.
	
	\bibitem[IZ21]{IseriZhang2021}
	Melih \.{I}\c{s}eri and Jianfeng Zhang.
	\newblock Set values for mean field games.
	\newblock {\em Preprint arXiv:2107.01661}, July 2021.
	
	\bibitem[IZ23]{IseriZhang2023}
	Melih \.{I}\c{s}eri and Jianfeng Zhang.
	\newblock Set valued {H}amilton-{J}acobi-{B}ellman equations.
	\newblock {\em Preprint arXiv:2311.05727}, November 2023.
	
	\bibitem[Jah04]{Jahn04}
	Johannes Jahn.
	\newblock {\em Vector Optimization - Theory, Applications, and Extensions.}
	\newblock Springer, 2004.
	
	\bibitem[KBL23]{KoEtAl2023}
	Hyungjin Ko, Junyoung Byun, and Jaewook Lee.
	\newblock A privacy-preserving robo-advisory system with the
	{B}lack-{L}itterman portfolio model: A new framework and insights into
	investor behavior.
	\newblock {\em Journal of International Financial Markets, Institutions and
		Money}, 89:101873, December 2023.
	
	\bibitem[KL12]{KuroiwaLee2012}
	Daishi Kuroiwa and Gue~Myung Lee.
	\newblock On robust multiobjective optimization.
	\newblock {\em Vietnam Journal of Mathematics}, 40:305--317, 2012.
	
	\bibitem[Kni21]{Knight1921}
	Frank~H. Knight.
	\newblock {\em Risk, {U}ncertainty, and {P}rofit}.
	\newblock Hart, Schaffner \& Marx; Houghton Mifflin Company, 1921.
	
	\bibitem[KR21]{KovacovaRudloff2019}
	Gabriela Kov{\'{a}}{\v{c}}ov{\'{a}} and Birgit Rudloff.
	\newblock Time consistency of the mean-risk problem.
	\newblock {\em Operations Research}, 69(4):1100--1117, July 2021.
	
	\bibitem[KRC22]{CKR2021}
	Gabriela Kov{\'{a}}{\v{c}}ov{\'{a}}, Birgit Rudloff, and Igor Cialenco.
	\newblock Acceptability maximization.
	\newblock {\em Frontiers of Mathematical Finance}, 1(2):219--248, 2022.
	
	\bibitem[KTW24]{KrachTeichmann2024}
	Florian Krach, Josef Teichmann, and Hanna Wutte.
	\newblock Robust utility optimization via a {GAN} approach.
	\newblock {\em Preprint arXiv:2403.15243}, March 2024.
	
	\bibitem[L{\"o}h11]{Loehne2011}
	Andreas L{\"o}hne.
	\newblock {\em Vector Optimization with Infimum and Supremum}.
	\newblock Springer Berlin Heidelberg, 2011.
	
	\bibitem[LR14]{LoehneRudloff2014}
	Andreas L\"ohne and Birgit Rudloff.
	\newblock An algorithm for calculating the set of superhedging portfolios in
	markets with transaction costs.
	\newblock {\em International Journal of Theoretical and Applied Finance},
	17(02):1450012, March 2014.
	
	\bibitem[Nau06]{Nau2006}
	Robert Nau.
	\newblock The shape of incomplete preferences.
	\newblock {\em The Annals of Statistics}, 34(5), October 2006.
	
	\bibitem[RU20]{RudloffUlus2020}
	Birgit Rudloff and Firdevs Ulus.
	\newblock Certainty equivalent and utility indifference pricing for incomplete
	preferences via convex vector optimization.
	\newblock {\em Mathematics and Financial Economics}, 15(2):397--430, October
	2020.
	
	\bibitem[Sha16]{Shapiro2016}
	Alexander Shapiro.
	\newblock Rectangular sets of probability measures.
	\newblock {\em Operations Research}, 64(2):528--541, April 2016.
	
	\bibitem[WD16]{WiecekDranichak2016}
	Margaret~M. Wiecek and Garrett~M. Dranichak.
	\newblock {\em Robust Multiobjective Optimization for Decision Making Under
		Uncertainty and Conflict}, pages 84--114.
	\newblock INFORMS, October 2016.
	
\end{thebibliography}

\newcommand{\etalchar}[1]{$^{#1}$}

\end{document}